\documentclass[a4paper,12pt,english]{amsart}
\usepackage{babel}
\usepackage[utf8]{inputenc}
\usepackage[T1]{fontenc}
\usepackage{amssymb,amsfonts,amsmath,amsthm}
\usepackage{color}
\usepackage{enumitem}
\usepackage{graphicx}
\usepackage{xspace}
\usepackage{mathtools}
\usepackage[stretch=16,shrink=16,step=2,kerning=true,protrusion=true,final]{microtype} 

\usepackage[colorlinks=true, linkcolor=blue, citecolor=red, urlcolor=blue]{hyperref}

\newtheorem{theorem}{Theorem}
\newtheorem{proposition}[theorem]{Proposition}
\newtheorem{corollary}[theorem]{Corollary}
\newtheorem{lemma}[theorem]{Lemma}

\theoremstyle{definition}
\newtheorem{definition}[theorem]{Definition}

\theoremstyle{remark}
\newtheorem{remark}[theorem]{Remark}

\newcommand{\ie}{i.e.\@\xspace}
\newcommand{\eg}{e.g.\@\xspace}
\newcommand{\cf}{c.f.\@\xspace}
\newcommand{\resp}{resp.\@\xspace}

\newcommand{\CC}{\mathbf{C}}

\newcommand{\RR}{\mathbf{R}}

\newcommand{\NN}{\mathbf{N}}

\newcommand{\C}{\CC}
\newcommand{\R}{\RR}

\newcommand{\N}{\NN}

\newcommand{\GL}{\operatorname{GL}}
\newcommand{\SL}{\operatorname{SL}}
\newcommand{\PSL}{\operatorname{PSL}}
\newcommand{\SO}{\operatorname{SO}}

\newcommand{\Sp}{\operatorname{Sp}}

\DeclareMathOperator{\Span}{Span}

\newcommand{\Ad}{\operatorname{Ad}}

\newcommand{\Id}{\operatorname{Id}}
\newcommand{\opp}{\mathrm{opp}}

\title[Generalizing Lusztig's total positivity II]{Generalizing Lusztig's total positivity II : geometric properties}

\author[O. Guichard]{Olivier Guichard}
\address{Universit\'e de Strasbourg, IRMA, 7 rue  Descartes, 67000 Strasbourg, France}
\email{olivier.guichard@math.unistra.fr}

\author[A. Wienhard]{Anna Wienhard}
\address{Max Planck Institute for Mathematics in the Sciences, Inselstr. 22, 04103 Leipzig, 
Germany}
\email{anna.wienhard@mis.mpg.de}

\thanks{
AW is supported by the European Research Council under ERC-Advanced Grant 101018839 and by the Hector Fellow Academy. She thanks the Institute for Advanced Study for its hospitality while some of this work was done.
AW thanks Misha Gekhtman for useful discussions and for raising the question whether the non-negative part of the flag variety is a closed ball.}

\AtBeginDocument{%
   \def\MR#1{}
}

\begin{document}

\numberwithin{theorem}{section} \numberwithin{equation}{section}

\begin{abstract}
Positive structures in Lie groups with respect to a subset $\Theta$ of the set of positive roots  provide a generalization of Lusztig's total positivity in split real Lie groups to the setting of general real semisimple Lie groups. In \cite{GW_pos} Lie groups $G$ admitting a positive structure were classified and many key properties of the unipotent positive semigroups were established. In this article we focus on the positive semigroup in $G$. We establish  key geometric properties of elements in the positive semigroup. Further we introduce corresponding positive and non-negative parts of flag varieties and determine the topology of the non-negative parts of flag varieties in many cases. For symplectic flag varieties we provide explicit descriptions of the positive and non-negative flag varieties. 
\end{abstract}

\maketitle

\tableofcontents

\section{Introduction}
Total positivity in $\mathrm{SL}(n,\mathbb{R})$ and more generally Lusztig's total positivity in split real Lie groups \cite{LusztigPosRed} plays an important role in different areas of mathematics, as well as in theoretical physics. Of particular interest are not just the totally positive semigroups in the Lie group, but also the positive and non-negative parts they define in flag varieties. The positive Grassmannian in particular has received a lot of attention due to its relevance for scattering amplitudes \cite{Arkani_etal_book}. 

In \cite{GW_pos} we introduced for semisimple real Lie groups $G$ the notion of positivity relative to a subset~$\Theta$ of the set of positive simple roots~$\Delta$. This provides a generalization of Lusztig's  total positivity, which is precisely positivity with respect the entire set~$\Delta$. In this case the group~$G$ has to be split. There are three further families of simple Lie groups admitting positive structures relative to proper subset~$\Theta$ of~$\Delta$: Hermitian Lie groups of tube type, indefinite orthogonal groups $\SO(p,q)$, $2< p<q$ and an exceptional family containing a real form of $F_4, E_6, E_7$ and $E_8$. 

For semisimple Lie groups $G$~endowed with a positive structure relative to~$\Theta \subset \Delta$, we introduced in \cite{GW_pos} the 
unipotent positive semigroup $U_{\Theta}^{>0}$ and a unipotent non-negative
semigroup $U_{\Theta}^{\geq 0}$ of the unipotent radical~$U_{\Theta}$ of the
standard parabolic subgroup~$P_{\Theta}$ determined
by $\Theta$.
Similarly, semigroups $U_{\Theta}^{\opp >0}$ and
$U_{\Theta}^{\opp \geq 0}$
in the opposite parabolic subgroup $P_{\Theta}^{\opp}$ were defined.
These semigroups give rise to the non-negative and positive semigroups $G_\Theta^{\geq 0}$ and $G_\Theta^{>0}$ in~$G$ \cite[Definition~4.2, Definition~4.4]{GW_pos}.

The general notion of positivity relative to a subset $\Theta$ provides an important framework for the study of higher Teichm\"uller spaces \cite{GLW}, \cite{Beyrer_Pozzetti}, \cite{BGLPW}, and is connected to the notion of magical triples introduced in \cite{Bradlow_Collier_etal}. 
The notion of positivity is also relevant beyond its connection to higher Teichmüller spaces, see \eg \cite{wienhardicm}. 
In this paper we focus on the study of the positive semigroup $G_{\Theta}^{> 0}$, the non-negative semigroups $G_{\Theta}^{\geq 0}$, and geometric properties of their elements,  as well as on the corresponding positive and non-negative parts in flag varieties $\mathsf{F}_{\Theta'}$ associated to $G$. In forthcoming work \cite{GW_alg} we focus on algebraic properties of positivity. Forthcoming work of the second author with Greenberg, Kaufman, and Niemeyer \cite{GKNW} develops cluster structures on semisimple Lie groups admitting a positive structure and contains several applications to algebraic properties of the positive semigroups as corollaries. 

\subsection*{Characterization of the positive semigroup}
The non-negative semigroup $G_{\Theta}^{\geq 0}$ is the semigroup generated by $U_{\Theta}^{\opp \geq 0}$, $U_\Theta^{\geq 0}$ and the connected component of the identity $L_{\Theta}^{\circ}$ of the Levi subgroup of~$P_\Theta$ and $P_{\Theta}^{\opp}$. The positive semigroup $G_{\Theta}^{> 0}$ is  the interior of $G_{\Theta}^{\geq 0}$.  
Our first result gives a characterization of  $G_{\Theta}^{> 0}$ and $G_{\Theta}^{\geq 0}$. 
\begin{theorem}\label{thm:positive_intro}
Let $G$ be a semisimple real Lie group admitting a positive structure relative to $\Theta$.
Then 
\begin{align*}
  G_{\Theta}^{>0} &= U_{\Theta}^{> 0} \cdot L_{\Theta}^{\circ} \cdot U_{\Theta}^{\opp > 0} = U_{\Theta}^{\opp >0} \cdot L_{\Theta}^{\circ} \cdot U_{\Theta}^{>0},\\
  G_{\Theta}^{\geq 0} &= U_{\Theta}^{\geq 0} \cdot L_{\Theta}^{\circ} \cdot U_{\Theta}^{\opp \geq 0} = U_{\Theta}^{\opp \geq 0} \cdot L_{\Theta}^{\circ} \cdot U_{\Theta}^{\geq 0}.
\end{align*} 
Furthermore, $G_{\Theta}^{\geq 0}$ is closed, and $G_{\Theta}^{\geq 0}G_{\Theta}^{>0} \subset G_{\Theta}^{>0}$. 
\end{theorem}
\begin{remark}
  As the product map \(U_\Theta\times L_{\Theta}\times U_{\Theta}^{\opp}\to G\) is one-to-one, an immediate consequence of the theorem is that the products maps \(U_{\Theta}^{>0}\times L_{\Theta}^{\circ}\times U_{\Theta}^{\opp>0}\to G^{>0}\) and \(U_{\Theta}^{\geq 0}\times L_{\Theta}^{\circ}\times U_{\Theta}^{\opp \geq 0}\to G^{\geq 0}\) are homeomorphisms.
  \end{remark}

We refer the reader to Section~\ref{sec:pos_group} for the proofs and a few further statements about the structure of the semigroups $G_{\Theta}^{> 0}$ and $G_{\Theta}^{\geq 0}$. 

\subsection*{Properties of elements in the positive semigroup}
The second result establishes key geometric properties of elements in the positive semigroup $G_{\Theta}^{> 0}$. 
In \cite{LusztigPosRed} Lusztig used canonical bases (\ie special bases of the irreducible representations of~\(G\) in which the images of the Chevalley basis of~\(\mathfrak{g}\) are sent to matrices with non-negative entries) to establish several important properties for elements~$g$ in the positive semigroup $G_\Delta^{> 0}$. We prove analogous properties for elements $g$ in $G_{\Theta}^{> 0}$ for all subsets $\Theta$, using mainly the action on flag varieties and properties established in \cite{GW_pos}.

Let us recall that an element $g$ in $G$ is said to be proximal with respect to $\Theta$ if $g$ has a unique attracting fixed point in the flag variety $\mathsf{F}_\Theta = G/P_\Theta$ and a unique repelling fixed point in $\mathsf{F}_{\Theta}^{\opp} = G/P_{\Theta}^{\opp}$. When $G$ admits a positive structure relative to $\Theta$ then $P_\Theta$ is conjugate to $P_{\Theta}^{\opp}$, thus $\mathsf{F}_\Theta = \mathsf{F}_{\Theta}^{\opp}$. 
\begin{theorem}\label{thm:positive_elem_intro}
Every element $g \in G_{\Theta}^{>0}$ is proximal with respect to~$\Theta$.  Furthermore, $g$~is conjugate to an element in~$L^{\circ}_{\Theta}$. 
\end {theorem}
The proximality of an element $g \in G_{\Theta}^{>0}$ was essentially proved in \cite{GLW}, even though the semigroup $G_{\Theta}^{>0}$ is not mentioned there.
 Here we in fact  prove a stronger property. 
For this let us recall that the Cartan projection of the element $g$ in $G$ is the element $\mu(g)$ in the closed Weyl chamber in the Cartan subspace $\mathfrak{a}$, which is uniquely determined by $g = k_1\mu(g)k_2$. 

\begin{theorem}\label{thm:positive_elem2_intro}
Every element $g \in G_{\Theta}^{>0}$ is conjugate to an element $\ell$ $L^{\circ}_\Theta$ which satisfies $\alpha(\mu(l)) >0$ for all $\alpha \in \Theta$. 
\end{theorem}
This is indeed a stronger statement. For example when $G = \Sp_{2n}(\R)$ and $\Theta = \{\alpha_n\}$, Theorem~\ref{thm:positive_elem_intro} only states that an element  $g \in G_{\Delta}^{>0}$ is conjugate to an element of $\GL_{n}(\R)$ of positive determinant. Whereas Theorem~\ref{thm:positive_elem2_intro} states that the singular values of that element are all bigger than one. 
In the split real case $\Theta = \Delta$, Theorem~\ref{thm:positive_elem2_intro} states 
that the element $g \in G_{\Delta}^{>0}$ is conjugate to an element in $A^\circ$ whose logarithm lies in the interior of the Weyl chamber, which means that all the entries are positive and distinct, giving strong regularity properties. When $G= \SL_n(\R)$ this is precisely the fact that a totally positive matrix in $\SL_n(\R)$ is diagonalizable with distinct positive eigenvalues. 

In the case of total positivity in split real Lie groups Theorem~\ref{thm:positive_elem2_intro} was proven in  \cite{LusztigPosRed}. Our general proof gives a new proof in the split real case which only uses the action of $G$ and $G_{\Theta}^{>0}$ on the flag variety~$\mathsf{F}_{\Theta}$. 
A crucial point that enables this approach is that in \cite{GW_pos} we do not only introduce the positive unipotent semigroup $U_\Theta^{>0}$, but we equally focus on the general notions of positive triple of flags and positive quadruples of flags. This allows us to investigate the action of~$G$ on these spaces of positive $n$-tuples of flags. In the split real case, the appropriate notions of positive configurations of $n$-tuples have been introduced in \cite{Fock_Goncharov}.  

\subsection*{Positive triples and quadruples of flags} 
The positive semigroup $U_\Theta^{>0}$ is  closely linked with the notion of positive triples, positive quadruples and more generally positive n-tuples in the flag variety $G/P_\Theta$, see \cite{GW_pos}, which is useful in applications of positivity, see \eg \cite{GLW, BGLPW, FTWZ}. 

The results on the positive semigroup $G_\Theta^{>0}$ have immediate corollaries for the positivity of quadruples. 
\begin{corollary}\label{cor:positive_intro}
Let $p_\Theta, {p}_{\Theta}^{\prime}$ be the points in~\(\mathsf{F}_{\Theta}\) corresponding to $P_\Theta$ and $P_\Theta^{\opp}$ respectively. Let~$g$ be an element in the positive semigroup $G_\Theta^{>0}$ and denote~\(g^+\), \(g^-\) its attracting and repelling fixed points in~\(\mathsf{F}_{\Theta}\). Then: 
\begin{enumerate}
\item The triples $({p}_\Theta, g\cdot {p}_\Theta, {p}_{\Theta}^{\prime})$ and $({p}_\Theta, g\cdot {p}_{\Theta}^{\prime}, {p}_{\Theta}^{\prime})$
 are positive. 
\item The quadruple $({p}_\Theta, g\cdot {p}_\Theta, {p}_{\Theta}^{\prime}, g^{-1}\cdot {p}^{\prime}_{\Theta})$ is positive. 
\item The quadruple \(({p}_\Theta, g\cdot {p}_\Theta, g\cdot{p}_{\Theta}^{\prime}, {p}^{\prime}_{\Theta})\) is positive. 
 \item\label{cor:item4} The quadruples $(g^-, {p}_{\Theta},g\cdot {p}_{\Theta},g^+)$ and $(g^-, {p}^{\prime}_{\Theta}, g\cdot {p}^{\prime}_{\Theta},g^+)$ are positive.
\end{enumerate}
\end{corollary}

%
  
In Section~\ref{sec:crossratio} we give a description of a cross ratio like function on positive quadruples which might be of independent interest.

\subsection*{Positive and Non-negative parts of flag varieties}
The positive and the non-negative semigroups $G_{\Theta}^{>0}$ and $G_{\Theta}^{\geq 0}$ allow us to introduce the corresponding positive and non-negative parts of flag varieties.

For an arbitrary subset  $\Theta' \subset \Delta$ let us consider the flag variety $\mathsf{F}_{\Theta'}= G/P_{\Theta'}$. We denote by~$p_{\Theta'}$ the base point corresponding to~$P_{\Theta'}$. The positive part $\mathsf{F}_{\Theta'} ^{>0}$ is the $G_{\Theta}^{>0}$-orbit of $p_{\Theta'}$, its negative part $\mathsf{F}_{\Theta'} ^{\geq 0}$ is the closure of $\mathsf{F}_{\Theta'}^{>0}$ in  $\mathsf{F}_{\Theta'}$.

When $\Theta' = \Theta$ these subsets have been described already in \cite{GW_pos} and used heavily in \cite{GLW}, where they are called diamonds. 
\begin{theorem}\label{thm:diamond_intro}
The positive part $\mathsf{F}_\Theta^{>0}$ is the standard diamond $D_\Theta$ with extremities $P_\Theta$ and $P_{\Theta}^{\opp}$. 
The non-negative part $\mathsf{F}_\Theta^{\geq 0}$ is the closure $\overline{D}_{\Theta}$ of the diamond $D_\Theta$.
\end{theorem}
The proof of Theorem~\ref{thm:positive_elem_intro} in fact follows easily from properties of diamonds established in \cite{GLW}. 

In order to describe the positive and non-negative part of general flag varieties, let us recall that, for any subsets \(\Upsilon\subset \Xi\) of~\(\Delta\), there is a natural projection 
\[\pi_{\Upsilon}^{\Xi}\colon  \mathsf{F}_{\Xi} \rightarrow \mathsf{F}_\Upsilon.\]

\begin{theorem}\label{thm:pospart_intro} 
Let $\mathsf{F}_{\Theta'}^{>0}$ be the positive part of the flag variety $\mathsf{F}_{\Theta'}$, and $\mathsf{F}_{\Theta'}^{\geq 0}$ the non-negative part. 
Then 
\begin{enumerate}[leftmargin=*]
\item\label{item1_thm:pospart_intro} If $\Theta'$~is contained in~$ \Theta$, then $\mathsf{F}_{\Theta'}^{>0} = \pi^{\Theta}_{\Theta'} (D_\Theta)$, and $\mathsf{F}_{\Theta'}^{\geq 0} = \pi^{\Theta}_{\Theta'} (\overline{D}_{\Theta})$.
In this case, the non-negative part $\mathsf{F}_{\Theta'}^{\geq 0}$ is homeomorphic to a closed ball. 
\item If \(\Theta'\) contains~$\Theta$, then $\mathsf{F}_{\Theta'}^{>0} = (\pi^{\Theta'})^{-1}_{\Theta} (D_\Theta)$, and $\mathsf{F}_{\Theta'}^{\geq 0} = {\pi^{\Theta'}}^{-1}_{\Theta} (\overline{D}_{\Theta})$.
In this case, the non-negative part $\mathsf{F}_{\Theta'}^{\geq 0}$ is homeomorphic to a compact fiber bundle over a  closed ball with fiber $K_{\Theta}/K_{\Theta'}$, where $K_\Theta$ and $K_{\Theta'}$ are the maximal compact subgroups of $L_\Theta$ and $L_{\Theta' }$ respectively.  In particular, when \(\Theta \subsetneq \Theta'\), $\mathsf{F}_{\Theta'}^{\geq 0}$ is topologically non-trivial. 
\item In general, one has the equalities $\mathsf{F}_{\Theta'}^{>0} = \pi^{\Xi}_{\Theta'} ((\pi^{\Xi}_{\Theta})^{-1}(D_\Theta))$, and $\mathsf{F}_{\Theta'}^{\geq 0} = \pi^{\Xi}_{\Theta'} ((\pi^{\Xi}_{\Theta})^{-1}(\overline{D}_{\Theta}))$, 
where $\Xi = \Theta \cup \Theta'$.
\end{enumerate}
\end{theorem}
For split real Lie groups and when $\Theta = \Delta$, we are always in situation~(\ref{item1_thm:pospart_intro}).  In that case, the result that the non-negative part of the flag variety is homeomorphic to a closed ball was proved by Galashin, Karp, and Lam \cite{GalashinKarpLam, GalashinKarpLam_Gr}. 
 The proof presented here applies to all semisimple Lie groups admitting a positive structure, in particular also to split real Lie groups \ie to the total positivity case; in that case it is simpler than the one in \cite{GalashinKarpLam, GalashinKarpLam_Gr} as it is more geometric in nature, and does not involve the use of canonical bases.   We only use the $\Theta$-principal $\mathfrak{sl}_2(\R)$, associated to the positive structure, and its action on the flag variety. 

There are three (families of) simple real Lie groups that admit two different positive structures with respect to two subsets \(\Theta\) and \(\Xi\). In that case one always has \(\Xi=\Delta\) so that the Lie group is split over~\(\R\). Thus, these are the split real groups with Dynkin diagrams $C_n$, $B_n$, or $F_4$.  We show in Section~\ref{sec:twostructures}  that in these cases the inclusions of $\mathsf{F}_{\Theta'}^{\Delta, >0} \subset \mathsf{F}_{\Theta'}^{>0}$ and $\mathsf{F}_{\Theta'}^{\Delta, \geq 0} \subset \mathsf{F}_{\Theta'}^{\geq 0}$ are strict.

In the case of split real Lie groups, there is an extensive literature that studies the non-negative parts of flag varieties, and describes its finer combinatorial structure and its cell decompositions. We propose to initiate a similar study for the non-negative parts of flag varieties with respect to positive structures relative to $\Theta$. In the case of the symplectic group $\Sp_{2n}(\R)$ and the closed diamonds $\overline{D}_{\{\alpha_n\}}$ in the space of Lagrangians a first discussion is given in \cite{Xie}, for $\SO(3,q)$ some aspects of the closed diamond $\overline{D}_{\{\alpha_1,\alpha_2\}}$ are described in \cite{GW_pos}

In Section~\ref{sec:symplectic}, we give a detailed discussion of the positive and non-negative parts of flag varieties for the symplectic group $\Sp_{2n}(\R)$ with its positive structure relative to  $\Theta = \{ \alpha_n\}$, where $\alpha_n$ is the short simple root. 
Of particular interest is the case of the projective space $\R\mathbb{P}^{2n-1}$. Here the non-negative part and the analogously defined non-positive part give a decomposition of $\R\mathbb{P}^{2n-1}$ into two non-linear half-spaces, which intersect along a quadric hypersurface. 

\begin{remark}
In many works on positive Grassmannians, the positive part is considered in an explicit affine chart, or with respect to an explicit basis. 
The point of view we take in \cite{GW_pos, GLW} and here is that we consider the entire family of diamonds, by considering not just the standard diamond, but all its translates. This is very fruitful as the nesting properties of diamonds provide insights on the dynamics of the action of $G$. Similarly, we expect that considering the families of positive parts of flag varieties will be very useful. For example, in the case of the symplectic group, the corresponding half-spaces in $\R\mathbb{P}^{2n-1}$ were used by \cite{Burelle_Treib} to give constructions of maximal Schottky groups. 
We propose to explore the families of positive parts of flag varieties with their nesting properties more systematically. 
\end{remark}

\section{Preliminaries and background}\label{sec:background}

Throughout the paper, $G$~is a real Lie group with finite center, finitely many connected components, and whose Lie algebra~\(\mathfrak{g}\) is semisimple.  The identity component of~\(G\) will be denoted by~\(G^\circ\).

\subsection{Classical notation}
\label{sec_class_notation}
Let $\tau\colon \mathfrak{g}\to \mathfrak{g}$ be a Cartan involution,
$\mathfrak{k}$ the fixed point set of~\(\tau\) (it is a maximal compact
subalgebra), and $\mathfrak{a} \subset \mathfrak{g}$ a Cartan subspace.  Let
\(K\subset G\) be the connected subgroup whose Lie algebra
is~\(\mathfrak{k}\); it is a maximal compact of \(G^\circ\).  We denote by $\Sigma$ the set of restricted roots, \ie
$\Sigma \subset \mathfrak{a}^*$ is the set of non-zero weights of the adjoint
action of $\mathfrak{a}$ on $\mathfrak{g}$.  Let $\Delta \subset \Sigma$ be a
set of simple roots.  We denote by~$W$ the Weyl group both seen as a subgroup of \(\GL(\mathfrak{a}^*)\) and as an abstract
Coxeter group (generated by reflections \(\{s_\alpha\}_{\alpha\in \Delta}\)).  Let
\(N_{K}(\mathfrak{a})\) be the normalizer of~\(\mathfrak{a}\) in~\(K\).  The natural map (induced by
the restriction to~\(\mathfrak{a}\) of the adjoint action)
\(N_{K}(\mathfrak{a})\to \GL(\mathfrak{a}^*)\) has image~\(W\).  For an
element~\(w\) in~\(W\), a lift of~\(w\) to \(N_{K}(\mathfrak{a})\) by
this map will be denoted by~\(\dot{w}\).  We denote by $w_\Delta$ the longest
element in~$W$, with respect to the word length on the generating set \(\{s_\alpha\}_{\alpha\in \Delta}\) and by $\iota\colon \Delta\to \Delta$ the opposition involution
(\ie for all~\(\alpha\in\Delta\), one has \(\iota(\alpha)=-w_\Delta(\alpha)\)).  The scalar product on~\(\mathfrak{a}^*\) induced by the Killing form on~\(\mathfrak{g}\) is denoted by~\((\cdot|\cdot)\).

The graph whose vertex set is~\(\Delta\), with an edge between~\(\alpha\) and~\(\beta\) if and only if \((\alpha|\beta)\neq 0\), and where an edge \((\alpha,\beta)\) is doubled (\resp tripled) when \(s_\alpha s_\beta\) is of order~\(4\) (\resp \(6\)) is called the \emph{Dynkin diagram} of~\(G\).

We will also sometimes use the group \(\mathrm{Aut}(\mathfrak{g})\).  It is a Lie subgroup of \(\GL(\mathfrak{g})\) and, as the Lie algebra~\(\mathfrak{g}\) is semisimple, the Lie algebra of \(\mathrm{Aut}(\mathfrak{g})\) is equal to~\(\mathfrak{g}\).  There is a natural action of \(\mathrm{Aut}(\mathfrak{g})\) on the Dynkin diagram of~\(\mathfrak{g}\) and the kernel of this action is denoted by \(\mathrm{Aut}_1(\mathfrak{g})\).

\subsection{Parabolic subgroups}
For any subset $\Theta \subset \Delta$ let $P_\Theta$  be the standard parabolic subgroup defined by $\Theta$ (with the convention that $P_\Delta$ is the minimal parabolic subgroup). More precisely ,\(P_\Theta\)~is the normalizer in~\(G\) of the Lie algebra \(\mathfrak{u}_\Theta = \sum_{{\alpha \in \Sigma_{\Theta}^+}}
\mathfrak{g}_{\alpha}\) (where $\mathfrak{g}_\alpha$ denotes the weight space of the root~$\alpha$).  Let~\(U_\Theta\) be the connected Lie subgroup whose Lie algebra is \(\mathfrak{u}_\Theta\).  The unipotent radical of~\(P_\Theta\) is~\(U_\Theta\) ; let~$\mathfrak{p}_\Theta$ be the Lie algebra of~\(P_\Theta\).  We denote by $P_{\Theta}^{\opp}$ the standard opposite  parabolic subgroup, by $U_{\Theta}^{\opp}$ its unipotent radical, and by $\mathfrak{p}_{\Theta}^{\opp}$ and $\mathfrak{u}_{\Theta}^{\opp}$ their Lie algebras.  One has \(\mathfrak{u}^{\opp}_\Theta = \sum_{{\alpha \in
  \Sigma_{\Theta}^+}} \mathfrak{g}_{-\alpha}\) 
The intersection $L_\Theta = P_\Theta \cap P_{\Theta}^{\opp}$ is a Levi component of~\(P_\Theta\) and of \(P_{\Theta}^{\opp}\), and its Lie algebra is denoted by $\mathfrak{l}_\Theta$.  

The Lie algebras of these subgroups admit the following decompositions: 
\[\mathfrak{l}_\Theta =\mathfrak{z}_{\mathfrak{g}}(\mathfrak{a}) \oplus 
\bigoplus_{\mathclap{\alpha \in \Span(\Delta \smallsetminus \Theta) \cap \Sigma^+}} (\mathfrak{g}_{\alpha} \oplus \mathfrak{g}_{-\alpha}),\ \mathfrak{p}_\Theta = \mathfrak{l}_\Theta \oplus \mathfrak{u}_\Theta, \text{ and } \mathfrak{p}_{\Theta}^{\opp} = \mathfrak{l}_\Theta \oplus \mathfrak{u}_{\Theta}^{\opp}.\]

\subsection{Structure of the Levi component} 
\label{sec_struct_levi}
The Lie algebra~$\mathfrak{l}_\Theta$ of the Levi component is invariant by the Cartan involution $\tau$. The intersection $\mathfrak{k}_\Theta = \mathfrak{k} \cap \mathfrak{l}_\Theta$ is the Lie algebra of a maximal compact connected subgroup~$K_\Theta$ of~$L_\Theta$. 
We have 
\begin{align}
\label{eq:max-compact-S_J}  \mathfrak{k}_\Theta &= \mathfrak{m} \oplus  
\bigoplus_{\mathclap{\alpha \in \Span(\Delta \smallsetminus \Theta) \cap \Sigma^+}} (\mathfrak{g}_{\alpha} \oplus \mathfrak{g}_{-\alpha}) \cap \mathfrak{k},\\
\label{eq:max-compact-S_J-factor}  (\mathfrak{g}_{\alpha} \oplus&
  \mathfrak{g}_{-\alpha}) \cap \mathfrak{k}  = \{ X+\tau(X)\}_{X\in \mathfrak{g}_\alpha},\ \forall \alpha \in  \Sigma^+.
\end{align}

The derived subgroup $S_\Theta= [L_\Theta, L_\Theta]$ of~$L_\Theta$ is a
semisimple\index{$S_\Theta$ the semisimple factor of $L_\Theta$}
real Lie group. 

We denote by \(L_{\Theta}^{\circ}\) the connected component of the identity
in~\(L_\Theta\) so that \(K_{\Theta}\)~is a maximal compact
subgroup of~\(L_{\Theta}^{\circ}\).  We will make use of the Cartan decomposition in~\(L_{\Theta}^{\circ}\):  setting 
\[\bar{\mathfrak{a}}_{\Theta}^{+} := \{X\in \mathfrak{a} \mid \forall \alpha\in \Delta\smallsetminus \Theta,\, \alpha(X)\geq 0\},\]
one has
\begin{itemize}
  \item For every~\(\ell\) in~\(L_{\Theta}^{\circ}\), there is a unique element \(X\in \bar{\mathfrak{a}}_{\Theta}^{+}\) such that \(\ell\) belongs to \(K_\Theta  \exp(X) K_\Theta \).
\end{itemize}

\subsection{Flag varieties}
\label{sec_flag_var}
Let $\Theta\subset \Delta$.  A subgroup of~$G$ which is conjugate to~$P_\Theta$ by the connected component~\(G^\circ\) of the identity in~$G$ is a \emph{parabolic subgroup} (of type~\(\Theta\)). A \emph{parabolic subalgebra} (of type~$\Theta$) of~$\mathfrak{g}$ is a
subalgebra conjugated (by an element of~\(G^\circ\)) to~$\mathfrak{p}_\Theta$. The space of all parabolic
subalgebras of type~$\Theta$ is called the \emph{flag variety} (of
type~$\Theta$) and denoted by~$\mathsf{F}_\Theta$.\index{$\mathsf{F}_\Theta$
  the flag variety associated with $P_\Theta$} The
space~$\mathsf{F}_\Theta$ is a subset of a Grassmannian variety of the real
vector space~$\mathfrak{g}$ and is endowed with the induced topology.
The point in~\(\mathsf{F}_\Theta\) corresponding to~\(\mathfrak{p}_\Theta\)
(base point) will be denoted by~\(p_\Theta\).
The flag
variety~$\mathsf{F}_\Theta$ is compact and the conjugation action by~$G^\circ$ on
 subalgebras 
  induces a continuous and transitive action
on~$\mathsf{F}_\Theta$;  $\mathsf{F}_\Theta$~identifies with
$G^\circ/(P_{\Theta}\cap G^{\circ})$ since $P_{\Theta} \cap G^{\circ}$ is the stabilizer of
$p_\Theta$ in~\(G^{\circ}\).
 
 The group~$G^\circ$ acts diagonally on $\mathsf{F}_\Theta \times
\mathsf{F}_{\iota(\Theta)}$. From the equality $\mathfrak{p}_{\Theta} = \Ad(\dot{w}_\Delta) \mathfrak{p}_{\iota(\Theta)}^{\opp}$,
the parabolic subalgebra~$\mathfrak{p}_{\iota(\Theta)}^{\opp}$ gives an
element of~$\mathsf{F}_{\Theta}$;
we will denote this element by~\(p^{\prime}_{\Theta}\).  The points \(p_{\iota(\Theta)}\) and \(p^{\prime}_{\iota(\Theta)}\) are base points in \(\mathsf{F}_{\iota(\Theta)}\).
A pair $(x,y)$ in $\mathsf{F}_\Theta \times
\mathsf{F}_{\iota(\Theta)}$ will be called \emph{transverse} if it belongs to
the $G^\circ$-orbit of $( {p}_\Theta,
{p}_{\iota(\Theta)}^{\prime})$. 
 \ We will also say that $x$~is transverse to~$y$.  By definition, there is one orbit of
transverse pairs; this orbit is isomorphic to $G^\circ/(G^\circ \cap L_\Theta)$ and is in fact the unique
open $G^\circ$-orbit in the product $\mathsf{F}_\Theta \times
\mathsf{F}_{\iota(\Theta)}$.

An element~\(g\) in~\(G^\circ\) acts thus on \(\mathsf{F}_\Theta\); \(g\)~is said to act proximally on~\(\mathsf{F}_\Theta\) if it has an attracting fixed point in~\(\mathsf{F}_\Theta\).  If it is the case, \(g\)~has a unique attracting fixed point~\(g^+\) in~\(\mathsf{F}_\Theta\), a unique repelling fixed point~\(g^-\) in \(\mathsf{F}_{\iota(\Theta)}\); the pair \((g^+,g^-)\) is transverse; the basin of attraction of~\(g^+\) is the affine chart of points that are transverse to~\(g^-\) (and the basin of repulsion of~\(g^-\) is the set of points transverse to~\(g^+\)).  If an element~\(g\) has a fixed point~\(x\) in~\(\mathsf{F}_\Theta\), then \(g\)~is proximal with attracting fixed point~\(x\) if and only if its tangential action of~\(g\) on \(T_x \mathsf{F}_\Theta\) is contracting (as a linear transformation of the vector space \(T_x \mathsf{F}_\Theta\)).

When $\iota(\Theta)=\Theta$ (which happens in particular when \(G\)~admits a positive structure with respect to~\(\Theta\)), one has \(\mathsf{F}_{\iota(\Theta)} = \mathsf{F}_\Theta\), \(p_{\iota(\Theta)} = p_\Theta\), \(p_{\iota(\Theta)}^{\prime} = p_{\Theta}^{\prime}\). The notion of transversality then makes sense for pairs of
elements in~$\mathsf{F}_\Theta$.

We denote the two orbit maps by
\begin{align}
  f&\colon G \to \mathsf{F}_\Theta \mid g \mapsto g\cdot {p}_\Theta\\
\intertext{and by }
  f^{\opp}&\colon G \to \mathsf{F}_{\iota(\Theta)} \mid g \mapsto  g\cdot {p}_{\iota(\Theta)}^{\prime}.
\end{align}
We will later use that $f$ and $f^{\opp}$  are open \(G\)-equivariant maps and that the restriction of~$f$ to $U_{\Theta}^{\opp}$ and the restriction of $f^{\opp}$ to $U_\Theta$ are diffeomorphisms onto their images \(f(U_{\Theta}^{\opp})\) and \(f^\opp(U_\Theta)\)  which are also equal to \(f(G)\) and \(f^\opp(G)\) respectively.

The affine chart of $\mathsf{F}_\Theta$ with respect to
${p}_{\iota(\Theta)}^{\prime}$ consists of all points transverse to
${p}_{\iota(\Theta)}^{\prime}$ and is precisely the image of~\(f\). The group $U_{\Theta}^{\opp}$ acts
simply transitive on the affine chart with respect to
${p}_{\iota(\Theta)}^{\prime}$ and $f|_{U_{\Theta}^{\opp}}$ is a \(L_\Theta\)-equivariant parametrization of the affine chart.  In the same way, the map \(X\mapsto \exp(X)\cdot p_\Theta\) from \(\mathfrak{u}_{\Theta}^{\opp}\) to the affine chart is a \(L_{\Theta}^{\circ}\)-equivariant parametrization. Similarly, for the affine chart with respect to~${p}_\Theta$. 

\subsection{Open Bruhat cells and transversality}
\label{sec_open_cell}

The open Bruhat cells are the open subset 
\[\Omega= (P_{\iota(\Theta)}\cap G^\circ) \dot{w}_\Delta (P_{\Theta}\cap G^\circ)\]  and \[\Omega^{\opp} = (P_{\iota(\Theta)}^{\opp}\cap G^\circ) \dot{w}_\Delta (P_{\Theta}^{\opp}\cap G^\circ),\] where $\dot{w}_{\Delta}$ is any representative in \(N_{K}( \mathfrak{a})\) of the longest element~$w_{\Delta}$ of the Weyl group~$W$.  
  
The equalities $\Omega= U_{\iota(\Theta)} \dot{w}_\Delta (P_{\Theta}\cap G^\circ)$  and $\Omega^{\opp} = U_{\iota(\Theta)}^{\opp} \dot{w}_\Delta (P_{\Theta}^{\opp}\cap G^\circ)$ also hold.  More precisely the map \[U_{\iota(\Theta)} \times (P_{\Theta}\cap G^\circ) \to \Omega\mid (g,h)\mapsto g \dot{w}_\Delta h\] is a diffeomorphism.

Note that, for similar reason, the product maps from \(U_{\Theta}^{\opp}\times L_\Theta \times U_\Theta\) to~\(G\) and from \(U_{\Theta}\times L_\Theta \times U_{\Theta}^{\opp}\) to~\(G\) are diffeomorphisms onto their images. However, these product maps are not proper since their images are not closed in~\(G\).

When \(\iota(\Theta)=\Theta\), one also has  $\Omega= U_{\Theta} \dot{w}_\Delta (P_{\Theta}\cap G^\circ)$  and $\Omega^{\opp} = U_{\Theta}^{\opp} \dot{w}_\Delta (P_{\Theta}^{\opp}\cap G^\circ)$

 Transversality of points in $\mathsf{F}_\Theta$ is connected to the open Bruhat cells. 
 An element $g \in G$ belongs to ~$\Omega$ (respectively to $\Omega^{\opp}$)
 if and only if $f(g)$ is transverse to ${p}_{\iota(\Theta)}$  (respectively if $f^{\opp}(g)$ is transverse to ${p}_{\Theta}^{\prime}$). 
 
 \subsection{Maps between flag varieties}\label{sec:maps}
    Let $\Theta\subset\Xi$ be two subsets of~\(\Delta\), then the flag variety  $\mathsf{F}_{\Xi}$ fibers over $\mathsf{F}_\Theta$.  In fact there is a unique $G^\circ$-equivariant map \(\mathsf{F}_\Xi\to \mathsf{F}_\Theta\) and we will denote if \(\pi_{\Theta}^{\Xi}\). Its fibers are compact and can be identified with flag varieties associated to the Levi subgroup~$L_{\Theta}^{\circ}$. 

These maps are compatible with each other: when \(\Theta\subset \Xi \subset \Xi'\), one has \(\pi_{\Theta}^{\Xi}\circ \pi_{\Xi}^{\Xi'} = \pi_{\Theta}^{\Xi'}\).

  When \(\Xi=\Delta\), we will simply denote by~\(\pi_\Theta\) the map \(\pi_{\Theta}^{\Delta}\colon \mathsf{F}_\Delta\to \mathsf{F}_\Theta\).

We will later make use of the following two facts: 
\begin{enumerate}
\item The fiber $(\pi^{\Xi}_\Theta)^{-1}({p}_\Theta)$ is connected.  Precisely the
  connected component of the identity $L_{\Theta}^{\circ}$ of $L_\Theta$ acts
  transitively on the fiber $(\pi^{\Xi}_{\Theta})^{-1}({p}_\Theta)$.
\item The map $\pi^{\Xi}_\Theta$ is a proper. In fact, it is even a Riemannian submersion when the flag varieties are equipped with \(K\)-invariant Riemannian structures. 
\end{enumerate}

Let~\(\Theta\) and~\(\Theta'\) be two subsets of~\(\Delta\) and let \(A\subset \mathsf{F}_\Theta\).  Due to the above compatibilities, one has that the subset \(\pi_{\Theta'}^{\Xi}\bigl( (\pi_{\Theta}^{\Xi})^{-1}(A)\bigr)\) does not depend on the subset~\(\Xi\) of~\(\Delta\) containing~\(\Theta\) and~\(\Theta'\) and is equal to \(\pi_{\Theta'}\bigl( (\pi_{\Theta})^{-1}(A)\bigr)\).

\subsection{The positive unipotent semigroup}
The Levi subgroup $L_\Theta$ acts via the adjoint action on
$\mathfrak{u}_\Theta$.  Let $\mathfrak{z}_\Theta$ denote the center of
$\mathfrak{l}_\Theta$ and $\mathfrak{t}_\Theta = \mathfrak{z}_\Theta \cap
\mathfrak{a}$ its intersection with the Cartan subspace.
Then
$\mathfrak{u}_\Theta$ decomposes into the weight 
spaces under the adjoint action of~$\mathfrak{t}_\Theta$; for every
 $\beta \in \mathfrak{t}^*_\Theta$, set
\[\mathfrak{u}_\beta\coloneq \{ N \in \mathfrak{g} \mid \mathrm{ad}(Z)N =
\beta(Z) N , \, \forall Z \in \mathfrak{t}_\Theta \}.
\]
These weight
spaces are of course related to those of~$\mathfrak{a}$:
\begin{equation}
\mathfrak{u}_\beta = \sum_{\mathclap{\substack{\alpha \in \Sigma\\ 
  \alpha|_{\mathfrak{t}_\Theta} = \beta }}}  \mathfrak{g}_{\alpha}. \label{eq:u_beta_sum_g_alpha}
\end{equation}
By 
abuse of notation, for an element~$\beta$
in~$\Sigma$, we will denote its restriction
to~$\mathfrak{t}_\Theta$  as well by~$\beta$, and hence by~$\mathfrak{u}_\beta$ the corresponding
weight space. In particular for  $\beta \in \Delta $, one has
\[\mathfrak{u}_\beta =\sum_{\mathclap{\substack{\alpha \in \Sigma\\ \alpha- \beta \in
  \Span(\Delta \smallsetminus \Theta)}}} \mathfrak{g}_{\alpha}. \]

The group~$G$ is said to \emph{admit a positive structure with respect} to~$\Theta$ \cite[Def.~3.1]{GW_pos} if, for every $\alpha \in \Theta$, there exists an acute nontrivial 
$L_\Theta^{\circ}$-invariant closed convex cone $c_\alpha \subset 
\mathfrak{u}_\alpha$; its interior~$\mathring{c}_\alpha$ is then non-empty.  The opposite cone $-c_\alpha$ is also $L_\Theta^{\circ}$-invariant. 
When $G$~admits a positive structure with respect to~$\Theta$, then necessarily $\iota(\Theta) = \Theta$ \cite[Rem.~3.5]{GW_pos}. 
We fix a cone~$c_\alpha$ for each~$\alpha$ in~$\Theta$, and the corresponding
opposite cone $c_{\alpha}^{\opp} = \tau(-c_\alpha) \subset
\mathfrak{u}_{\alpha}^{\opp}$ (\(\tau\)~is the Cartan involution, see Section~\ref{sec_class_notation}).  This amounts to a $U$-pinning \cite[Section~2.1]{GLW}. We refer the reader to \cite{GW_pos} for more details. 

The non-negative unipotent semigroup $U_\Theta^{\geq 0}$ is the semigroup of
$U_\Theta$ generated by $\exp(v)$ with $v \in c_\alpha$, $\alpha \in
\Theta$. 
 The positive semigroup $U_{\Theta}^{>0}$ is the interior of $U_{\Theta}^{\geq 0}$.  The positive semigroup $U_{\Theta}^{\opp >0}$ and the opposite non-negative unipotent semigroup $U_{\Theta}^{\opp \geq 0}$ inside $U_{\Theta}^{\opp}$ are defined analogously:  \(U_{\Theta}^{\opp \geq 0}\) is the semigroup generated by \(\exp\bigl( \bigcup_{\alpha\in \Theta} c_{\alpha}^{\opp}\bigr)\) and \(U_{\Theta}^{\opp > 0}\) is its interior.

The non-negative and positive unipotent semigroups  lead to the definition of the non-negative and positive semigroups in~$G$. 

\begin{definition}\label{def:nonnegative}\cite[Def.~4.2]{GW_pos}
Let $G$ be a semisimple real Lie group admitting a positive structure with respect~$\Theta$. The \emph{non-negative semigroup} $G_{\Theta}^{\geq 0}$ in~$G$ is the semigroup generated by $U_{\Theta}^{\geq 0}$, $L_{\Theta}^{\circ}$, and $U_{\Theta}^{\opp \geq 0}$. 
The \emph{positive semigroup} $G_{ \Theta}^{>0}$ is the interior of $G_{\Theta}^{\geq 0}$. 
\end{definition}

\subsection{Diamonds}
\label{sec_diamonds}
The unipotent positive semigroup gives rise to the notion of positive triples, and more generally positive $n$-tuples in the flag variety $\mathsf{F}_{\Theta}$. To organize these notions, we use in \cite[Section~13]{GW_pos} the notion of diamonds (this notion was introduced in \cite{LabourieToulisse2023} for the case of \(\SO(2,n)\) and exploited in \cite{GLW} in full generality). 
The standard diamond defined by $U_\Theta^{>0}$ is the set 
\[D_\Theta = U_{\Theta}^{>0}\cdot {p}_{\Theta}^{\prime} \subset \mathsf{F}_{\Theta}.\] 
Furthermore, the equalities \(D_\Theta = U_{\Theta}^{\opp >0}\cdot {p}_\Theta = f(U_{\Theta}^{\opp >0}) = f^{\opp} (U_\Theta^{>0})\) hold.

A diamond with extremities $a,b \in \mathsf{F}_{\Theta}$ is the image
of $D_\Theta$ under an element of $\mathrm{Aut}_1(\mathfrak{g})$ that sends
$({p}_\Theta, {p}_{\Theta}^{\prime})$ to $(a,b)$; we sometimes include the extremities of the diamond in the notation: $D(a,b)$. There are several
diamonds with extremities $a,b$ \cite[Cor.~13.5]{GW_pos}.  However, for any given diamond~$D$ with extremities~$a$ and~$b$ there is a unique opposite diamond with extremities~$a$ and~$b$ denoted~$D^\vee$: the diamond opposite to~\(D_\Theta\) is \(D_{\Theta}^{\vee}= (U_{\Theta}^{>0})^{-1}\cdot p_{\Theta}^{\prime}\) and if \(D=g\cdot  D_\Theta\) (with \(g\in \mathrm{Aut}_1(\mathfrak{g})\)), one has \(D^\vee = g\cdot D_{\Theta}^{\vee}\) (this does not depend on the choices \cite[Lem.~13.9]{GW_pos}).  Let~\(x\) be in~\(D\), then \(D^\vee\)~is in fact the set of \(y\in \mathsf{F}_\Theta\) such that \((y,a,x,b)\) is a positive quadruple.

Let us recall a couple of facts from \cite[Section~13]{GW_pos} and \cite[Section~2]{GLW}.  
\begin{enumerate}[leftmargin=*]
\item A triple  $(f_1, f_2, f_3)$ of points in $\mathsf{F}_{\Theta}$ is positive if $f_1$ and $f_3$ are transverse and there is a diamond with extremities $f_1,f_3$ that contains~$f_2$. 
\item A quadruple $(f_1, f_2, f_3, f_4)$ of points in $\mathsf{F}_{\Theta}$ is positive if and only if there exists a diamond $D$ with extremities $f_1,f_3$ that contains~$f_2$ and such that $f_4$~belongs to~$D^\vee$. 
\item Let $(f_1, f_2, f_3, f_4)$ be a positive quadruple, then there exist a unique pair of diamonds $D(f_1, f_4)$ and $D(f_2, f_3)$ such that  ${D(f_2, f_3)}$ is
 contained in  $D(f_1, f_4)$; furthermore the closure $\overline{D}(f_2, f_3)$ is
 contained in  $D(f_1, f_4)$. 
\end{enumerate}

\subsection{The $\Theta$-principal $\mathfrak{s}\mathfrak{l}_2$} 
\label{sec_theta_prin}
In \cite[Section~5.5]{GW_pos} we introduced the notion of $\Theta$-system. 
 A $\Theta$-system $(E_\alpha, F_\alpha, D_\alpha)_{\alpha\in \Theta}$ consists of a family of special elements $E_\alpha \in \mathring{c}_\alpha$, $F_\alpha \in \mathring{c}^{\opp}_{\alpha}$, $D_\alpha = [E_\alpha, F_\alpha]$ for \(\alpha\in \Theta\), which are constructed from the root vectors $X_\alpha \in \mathfrak{g}_\alpha$, and their orbits under the Weyl group of~\(S_\Theta\).

More precisely, in most cases, one has \(\mathfrak{u}_\alpha= \mathfrak{g}_\alpha\) in which case \(E_\alpha=X_\alpha\) is a weight vector for the action of~\(\mathfrak{a}\) (and \(c_\alpha = \R_{\geq 0} X_\alpha\)).  In some case one has \(\mathfrak{u}_\alpha\neq \mathfrak{g}_\alpha\) (in \cite{GW_pos} such a root~\(\alpha\) is called a ``special root''), in which case there is a family of pairwise orthogonal roots \((\gamma_i)_{i=0, \dots, d}\) with \(\gamma_0=\alpha\) and, for all~\(i\) the weight \(\gamma_i-\alpha\) is a non-negative combination of simple roots in \(\Delta\smallsetminus \Theta\), one has \(E_\alpha = \sum_{i=0}^{d} X_{\gamma_i} = X_\alpha + \sum_{i=1}^{d} X_{\gamma_i}\), and a linear combination \(\sum_{i=0}^{d} \lambda_i X_{\gamma_i}\) belongs to \(\mathring{c}_\alpha\) if and only if, for all~\(i\), one has~\(\lambda_i>0\).

A $\Theta$-system generates a split real subalgebra $\mathfrak{h}_\Theta$ in $\mathfrak{g}$. The principal \(3\)-dimensional subalgebra of $\mathfrak{h}_\Theta$  is called the $\Theta$-principal $\mathfrak{s}\mathfrak{l}_2$ in $\mathfrak{g}$. It is generated by elements $E = \sum_{\alpha\in\Theta} q_{\alpha}^{1/2} E_\alpha,F = \sum_{\alpha\in\Theta} q_{\alpha}^{1/2} F_\alpha,D = \sum_{\alpha\in\Theta} q_{\alpha} D_\alpha$ for an explicitly defined family of positive integers \((q_\alpha)_{\alpha\in\Theta}\).  
We denote by $\mathsf{F}_0$ the flag variety of this \(3\)-dimensional simple algebra.  Its base points, corresponding to the parabolic subalgebras $\R D\oplus \R E$ and $\R D\oplus \R F$ are denoted~\(p\) and~\(p'\).  They are the extremities of the standard diamond~\(D_{0}\) of \(\mathsf{F}_0 \simeq \R\mathbb{P}^1\), as well as the extremities of the standard opposite diamond \(D_{0}^{\vee}\). 
The $\Theta$-principal $\mathfrak{s}\mathfrak{l}_2$ gives rise to an equivariant embedding $\phi \colon \mathsf{F}_0 \rightarrow \mathsf{F}_{\Theta}$, with the properties that $\phi(p) = {p}_{\Theta}$, $\phi(p') = {p}_{\Theta}^{\prime}$, 
 $\phi(D_{0}) \subset D_\Theta$, and $\phi(D_{0}^{\vee}) \subset D_\Theta^\vee$. 

Furthermore, for all~\(s\) in~\(\R\), the element \(x(s)=\exp(s E)\) of~\(G\) belongs to~\(U_{\Theta}^{>0}\) if and only if \(s>0\).  Similarly, the element \(y(s)=\exp(s F)\) belongs to \(U_{\Theta}^{\opp>0}\)  if and only if \(s>0\).  Equally, for all~\(\lambda\in \R_{>0}\), the element \(\check{\chi}(\lambda)= \exp\bigl( (\log \lambda)D\bigr)\) belongs to~\(L_{\Theta}^{\circ}\).

We will later need the following exchange relation.

\begin{lemma}\label{lem_braid_sl2}
  For all~\(t\) in~\(\R\), one has
  \[x(t)y(t) = y\Bigl( \frac{t}{1+t^2}\Bigr) \check{\chi}(1+t^2) x\Bigl( \frac{t}{1+t^2}\Bigr).\]
\end{lemma}
\begin{proof}
  Since the relation only involves the principal \(\mathfrak{sl}_2\), it is sufficient to work in the connected subgroup of~\(G\) whose Lie algebra is the \(\Theta\)-principal \(\mathfrak{sl}_2(\R)\).  This subgroup is a cover of \(\PSL_2(\R)\) and it is therefore enough to work in the universal cover \(\widetilde{\SL}_2(\R)\).
  
  Elements of \(\widetilde{\SL}_2(\R)\) are represented by continuous paths \((g(s))_{s\in [0,1]}\) in \(\SL_2(\R)\) and the multiplication in \(\widetilde{\SL}_2(\R)\) is the componentwise multiplication of paths : \((g(s))_{s\in [0,1]} (h(s))_{s\in [0,1]} = (g(s) h(s))_{s\in [0,1]}\).  The element \(x(t)\) of \(\widetilde{\SL}_2(\R)\) is represented by \(\bigl( \bigl(\begin{smallmatrix}
    1 & st \\ 0 & 1
  \end{smallmatrix}\bigr)\bigr)_{s\in [0,1]}\) and similarly for \(y(t)\).  The element \(\check{\chi}(1+t^2)\) is represented by \(\bigl( \bigl(\begin{smallmatrix}
    1+s^2 t^2 & 0 \\ 0 & (1+s^2 t^2)^{-1}
  \end{smallmatrix}\bigr)\bigr)_{s\in [0,1]}\).  Hence, the result will follow from the following relation in \(\SL_2(\R)\):
  \[\begin{pmatrix}
    1 & st \\ 0 & 1
  \end{pmatrix}\begin{pmatrix}
    1 & 0 \\ st & 1
  \end{pmatrix} = \begin{pmatrix}
    1 & 0 \\ \frac{st}{1+s^2t^2} & 1
  \end{pmatrix}\begin{pmatrix}
    1+s^2 t^2 & 0 \\ 0 & \frac{1}{1+s^2t^2}
  \end{pmatrix}
  \begin{pmatrix}
    1 & \frac{st}{1+s^2t^2} \\ 0 & 1
  \end{pmatrix} \]
  which is a direct computation.
\end{proof}

\section{The positive semigroup in $G$}\label{sec:pos_group}
In this section we investigate the finer structure of the non-negative and the positive semigroups $G_{\Theta}^{\geq 0}$ and~$G_{\Theta}^{>0}$. 

The main statement is the following description of the positive semigroup.

\begin{theorem}\label{thm:decomposition}
Let $G$ be a semisimple real Lie group admitting a positive structure relative to $\Theta$.
Then the positive semigroup $G_{\Theta}^{>0}$ and the non-negative semigroup $G_{\Theta}^{\geq 0}$ admit the following decompositions  
\begin{align*}
  G_{\Theta}^{>0} &= U_{\Theta}^{> 0} \cdot L_{\Theta}^{\circ} \cdot U_{\Theta}^{\opp > 0} = U_{\Theta}^{\opp >0} \cdot L_{\Theta}^{\circ} \cdot U_{\Theta}^{>0},\\
  G_{\Theta}^{\geq 0} &= U_{\Theta}^{\geq 0} \cdot L_{\Theta}^{\circ} \cdot U_{\Theta}^{\opp \geq 0} = U_{\Theta}^{\opp \geq 0} \cdot L_{\Theta}^{\circ} \cdot U_{\Theta}^{\geq 0}.
\end{align*} 
\end{theorem}

\begin{remark}
  Since the product map on \(U_{\Theta} \times L_{\Theta}^{\circ} \times U_{\Theta}^{\opp }\) (\resp on \(U_{\Theta}^{\opp } \times L_{\Theta}^{\circ} \times U_{\Theta}\)) is one-to-one (\cf Section~\ref{sec_open_cell}), this result gives also parametrizations of the semigroups \(G^{>0}_{\Theta}\) and \(G_{\Theta}^{\geq 0}\).
\end{remark}

\begin{remark}
One could alternatively introduce the semigroup as the set  \(U_\Theta^{> 0} \cdot L_{\Theta}^{\circ} \cdot U_{\Theta}^{\opp > 0}\) and then show that this is a semigroup. This is the strategy of Lusztig in \cite{LusztigPosRed}. This also works for positivity relative to~$\Theta$, and is useful when one does not work over~$\R$. One feature of our approach is that we do not have to prove a general exchange relation for $G$, but only need the exchange relation in \(\SL_2(\R)\), see Lemma~\ref{lem_braid_sl2}. Let us note that in \cite{GKNW} the general exchange relation is realized as a sequence of cluster mutations in a non-commutative cluster algebra. 
\end{remark}

The key steps in establishing Theorem~\ref{thm:decomposition} (or similarly in showing that  \(U_\Theta^{> 0} \cdot L_{\Theta}^{\circ} \cdot U_{\Theta}^{\opp > 0}\) is a semigroup) are formulated in the next theorem. 

\begin{theorem}\label{thm:positive}
Let $G$ be a semisimple real Lie group admitting a positive structure with respect to~$\Theta$.

 Consider the subset $S=U_\Theta^{> 0} \cdot L_{\Theta}^{\circ} \cdot U_{\Theta}^{\opp > 0}$. Then the following holds 
  \begin{enumerate}
 \item\label{item1_thm:positiie} For every $s\in S$, the points \(s\cdot p'_\Theta\) and \(s\cdot p_\Theta\) belong to the standard diamond~\(D_\Theta\). 
 \item \label{item1bis_thm:positive} The closure~\(\overline{S}\) is contained in \(U_{\Theta}^{\geq 0} \cdot L_{\Theta}^{\circ} \cdot U_{\Theta}^{\opp\geq 0}\).
 \item\label{item2_thm:positive} $S$ is a connected component of $\Omega\cap \Omega^{\opp}$ (Section~\ref{sec_open_cell} introduces the open Bruhat cells~$\Omega$ and~$\Omega^{\opp}$). 
 \item\label{item3_thm:positive} One has $S = U_{\Theta}^{\opp >0} \cdot L_{\Theta}^{\circ} \cdot U_{\Theta}^{>0}$.
 \item\label{item4_thm:positive} One has $ G_{\Theta}^{\geq 0} S \subset S$.  
 \item\label{item5_thm:positive} The set~$S$ is a semigroup.
 \item\label{item6_thm:positive} The non-negative semigroup~$G_{\Theta}^{\geq 0}$ is the closure of~$S$.
 \item\label{item7_thm:positive} One has $G_{\Theta}^{\geq 0} = U_\Theta^{\geq 0} \cdot L_{\Theta}^{\circ} \cdot U_{\Theta}^{\opp \geq 0} =  U_{\Theta}^{\opp \geq 0} \cdot L_{\Theta}^{\circ}\cdot U_\Theta^{\geq 0}$. 
 \item \label{item8_thm:positive} The following equalities hold $G_{\Theta}^{\geq 0}\cap (\Omega\cap \Omega^{\opp}) = S $ and \(S=G_{\Theta}^{>0}\).
 \end{enumerate} 
\end{theorem}

The proof of Theorem~\ref{thm:positive} relies on similar properties for the positive unipotent semigroups, that were established in \cite[Section~10]{GW_pos}. 
\begin{proof}
We prove the statements one by one. 

\begin{enumerate}[leftmargin=0pt]
\item Let $s = u\ell v$ with $u\in U_\Theta^{> 0} $, $\ell \in L_{\Theta}^{\circ}$, $v \in U_{\Theta}^{\opp > 0}$. Then $s \cdot p'_\Theta = u \cdot p'_\Theta$ lies in $D_\Theta$ since $D_\Theta = U_\Theta^{> 0} \cdot p'_\Theta$. 
From the equality $D_\Theta = U_{\Theta}^{\opp > 0} \cdot  p_\Theta$, we have $v \cdot p_\Theta \in D_\Theta$; since $D_\Theta$ is $L_{\Theta}^{\circ}$-invariant, we have also $\ell v\cdot p_\Theta \in D_\Theta$. Thus, 
$\ell v\cdot p_\Theta = w\cdot p'_\Theta$ for some $w \in  U_\Theta^{> 0}$, and since $U_\Theta^{> 0}$ is a semigroup we conclude that $u\ell v \cdot p_\Theta = uw\cdot p'_\Theta$ lies in $D_\Theta$. 

\item Let \(g\) be in \(\overline{S}\).  There exists \((s_n)_{n\in \N}\) a sequence in~\(S\) converging to~\(g\).  The sequence \((s_n\cdot p_\Theta)_{n\in\N}\) converges to \(g\cdot p_\Theta\).  Since, for all~\(n\), \(s_n\cdot p_\Theta\) belongs to the standard diamond~\(D_\Theta\), one has hence that \(g\cdot p_\Theta\) belongs to \(\overline{D}_\Theta\).  Similarly, \(g\cdot p_{\Theta}^{\prime}\) belongs to \(\overline{D}_\Theta\).

Let~\(x\) and~\(y\) be in the standard opposite diamond \(D_{\Theta}^{\vee}\) and such that the quadruple \((p_{\Theta}^{\prime}, x,y, p_\Theta)\) is positive.  Let \(D(x,y)\) be the diamond with extremities~\(x\), \(y\) containing \(D_\Theta\).  One has also \(\overline{D}_\Theta\subset D(x,y)\) (\cf Section~\ref{sec_diamonds})
and thus the elements \(g\cdot p_\Theta\) and \(g\cdot p_{\Theta}^{\prime}\) belongs to \(D(x,y)\).  This means that the triples \((x, g\cdot p_{\Theta}, y)\) and \((x, g\cdot p_{\Theta}^{\prime}, y)\) are positive.  There are thus unique diamonds \(D(x, g\cdot p_\Theta)\) and \(D(x, g\cdot p_{\Theta}^{\prime})\) with extremities \(x\), \(g\cdot p_{\Theta}\) (respectively \(x\), \(g\cdot p_{\Theta}^{\prime}\)) and contained in \(D(x,y)\).

Likewise, for all~\(n\) in~\(\N\), there are unique diamonds \(D(x, s_n\cdot p_\Theta)\) and \(D(x, s_n\cdot p_{\Theta}^{\prime})\) with extremities \(x\), \(s_n\cdot p_{\Theta}\) (respectively \(x\), \(s_n\cdot p_{\Theta}^{\prime}\)) and contained in \(D(x,y)\).  One has 
 \(D(x, g\cdot p_\Theta) = \lim_{n\to \infty} D(x, s_n \cdot p_\Theta)\) and \(D(x, g\cdot p_{\Theta}^{\prime}) = \lim_{n\to \infty} D(x, s_n \cdot p_{\Theta}^{\prime})\).

Since, for all~\(n\), the sextuple \((x, p_\Theta,s_n\cdot p_\Theta,s_n\cdot p_{\Theta}^{\prime},p_{\Theta}^{\prime}, y)\) is positive,
the quadruple \((x,s_n\cdot p_\Theta,s_n\cdot p_{\Theta}^{\prime}, y)\) is positive
and \(D(x, s_n\cdot p_\Theta) \subset D(x, s_n\cdot p_{\Theta}^{\prime})\).  Taking the limit, one gets  \(\overline{D}(x, g\cdot p_\Theta) \subset \overline{D}(x, g\cdot p_{\Theta}^{\prime})\) and, taking the interiors of these diamonds, one gets \(D(x, g\cdot p_\Theta) \subset D(x, g\cdot p_{\Theta}^{\prime})\); this means that the quadruple \((x, g\cdot p_\Theta, g\cdot p_{\Theta}^{\prime}, y)\) is positive.

For all~\(n\) in~\(\N\), the point~\(p_\Theta\) belongs to \(D(x, s_n\cdot p_\Theta)\).  Therefore \(p_\Theta\)~belongs to \(\overline{D}(x, g\cdot p_\Theta)\).  Furthermore, since  the quadruple \((x, g\cdot p_\Theta, g\cdot p_{\Theta}^{\prime}, y)\) is positive, the point \(g\cdot p_{\Theta}^{\prime}\) belongs to the opposite diamond \(D^{\vee}(x, g\cdot p_\Theta)\).  These facts together imply that \(p_\Theta\)~is transverse to~\(g\cdot p_{\Theta}^{\prime}\).

There is hence a unique~\(u\) in \(U_{\Theta}\) such that \(u\cdot p_{\Theta}^{\prime} = g\cdot p_{\Theta}^{\prime}\).  Thus, \(u^{-1}g\) belongs to \(P_{\Theta}^{\opp}= L_\Theta U_{\Theta}^{\opp}\): there are unique elements \(\ell\in L_\Theta\) and \(v\in U_{\Theta}^{\opp}\) such that \(g=u\ell v\).  For all~\(n\) in~\(\N\), let us denote \((u_n, \ell_n, v_n)\) the unique element of \(U_{\Theta}^{>0} \times L_{\Theta}^{\circ} \times U_{\Theta}^{\opp>0}\) such that \(s_n = u_n \ell_n v_n\).  Since, for all~\(n\), one has \(s_n\cdot p_{\Theta}^{\prime} = u_n \cdot p_{\Theta}^{\prime}\), since the sequence \(( s_n\cdot p_{\Theta}^{\prime})_{n\in \N}\) converges to \(g\cdot p_{\Theta}^{\prime}= u\cdot p_{\Theta}^{\prime}\), one gets that the sequence \((u_n \cdot p_{\Theta}^{\prime})_{n\in\N}\) converges to \(u\cdot p_{\Theta}^{\prime}\).
 Since the map \(U_\Theta \to \mathsf{F}_\Theta\mid u\mapsto u\cdot p_{\Theta}^{\prime}\) is a diffeomorphism onto its image (this is the map \(f^{\opp}|_{U_\Theta}\) of Section~\ref{sec_flag_var}), one obtains that the sequence \((u_n)_{n\in\N}\) converges to~\(u\).  This implies that \(u\)~belongs to \(U_{\Theta}^{\geq 0}\).  Looking at the sequence 
\((s_{n}^{-1}\cdot p_\Theta)\) one deduces similarly that the sequence \((v_{n}^{-1})_{n\in\N}\) converges to~\(v^{-1}\), hence that the sequence \((v_{n})_{n\in\N}\) converges to~\(v\) and that \(v\)~belongs to \(U_{\Theta}^{\opp\geq 0}\).  From that it follows that the sequence \((\ell_n)_{n\in\N}\) converges to~\(\ell\) and that the element~\(\ell\) belongs to~\(L_{\Theta}^{\circ}\).  

We obtained that \(g\) belongs to~\(U_{\Theta}^{\geq 0} L_{\Theta}^{\circ} U_{\Theta}^{\opp\geq 0}\) and thus that \(\overline{S}\)~is contained in \(U_{\Theta}^{\geq 0} L_{\Theta}^{\circ} U_{\Theta}^{\opp\geq 0}\).

\item Let~\(g\) be an element in~$S$.  There are $u \in U_\Theta^{> 0}$, $\ell\in L_{\Theta}^{\circ}$, and $v \in U_{\Theta}^{\opp >0}$ such that  $g = u \ell v$.

Recall that $g \in \Omega$ if and only if $f^{\opp}(g)$ is transverse to ${p}_{\Theta}^{\prime}$, and $g \in \Omega^{\opp}$ if and only if $f(g)$ is transverse to ${p}_\Theta$ (the maps~\(f\) and \(f^\opp\) have been introduced in Section~\ref{sec_flag_var}).
Let us consider $f^{\opp}(g) = f^{\opp}(u\ell v) = u\cdot {p}_{\Theta}^{\prime}$.  Since $u \in U_{\Theta}^{> 0}$, this point is transverse to ${p}_{\Theta}^{\prime}$ by \cite[Corollary~10.19]{GW_pos}.  Hence \(g\)~belongs to~$\Omega$.  Similarly \(g\)~belongs to~$\Omega^{\opp}$. 
 
Thus, $S \subset \Omega\cap \Omega^{\opp}$. Since $S$~is connected, it is contained in one connected component.  Furthermore, since the product map is a diffeomorphism on its image and has full rank (\cf Section~\ref{sec_open_cell}), \(S\)~is open in~\(G\) and it is thus open in \(\Omega\cap \Omega^{\opp}\).
So we only have to show that $S$~is closed in $\Omega\cap \Omega^{\opp}$ to deduce that $S$~is a full connected component.

Let thus~\(g\) be in \(\overline{S}\cap \Omega \cap \Omega^\opp\).  Since \(g\in \overline{S}\), by point~(\ref{item1bis_thm:positive}) there is \(u\in U_{\Theta}^{\geq 0}\), \(\ell \in L_{\Theta}^{\circ}\), and \(v\in U_{\Theta}^{\opp \geq 0}\) such that \(g=u \ell v\).  Since \(g\in \Omega^\opp\), \(g\cdot p_{\Theta}^{\prime}\) is transverse to~\(p_{\Theta}^{\prime}\).  Since \(g\cdot p_{\Theta}^{\prime} = u\cdot p_{\Theta}^{\prime}\), we obtained that \(u\)~belongs to \(U_{\Theta}^{>0}\) \cite[Section~13.1]{GW_pos}.  Similarly, \(v\)~belongs to \(U_{\Theta}^{\opp >0}\) and we conclude that \(g\)~belongs to~\(S\).

\item 
Applying the same argument as in the proof of~(\ref{item2_thm:positive}) to the subset $S^{\opp}= U_{\Theta}^{\opp >0} \cdot L_{\Theta}^{\circ} \cdot U_{\Theta}^{>0}$, we obtain that $S^{\opp}$ is also a connected component of $\Omega\cap \Omega^{\opp}$.  We thus only have to show that \(S\) and~\(S^\opp\) intersect; 
this is provided by the equality of Lemma~\ref{lem_braid_sl2} since, for \(t>0\), the elements \(x(t)\), \(y(t)\), and \(\check{\chi}(t)\) belong to \(U_{\Theta}^{>0}\), \(U_{\Theta}^{\opp >0}\), and \(L_{\Theta}^{\circ}\) respectively.

\item 
    Since \(G^{\geq 0}\) is generated by \(U_{\Theta}^{\geq 0}\), \(L_{\Theta}^{\circ}\), and \(U_{\Theta}^{\opp\geq 0}\), it is enough to prove the inclusions \(U_{\Theta}^{\geq 0} S \subset S\), \(L_{\Theta}^{\circ} S \subset S\), and \(U_{\Theta}^{\opp\geq 0} S \subset S\).  The first two inclusions are direct consequences of \(U_{\Theta}^{\geq 0}U_{\Theta}^{> 0} \subset U_{\Theta}^{> 0}\) \cite[Theorem~1.4]{GW_pos} and the stability of \(U_{\Theta}^{\geq 0}\) by conjugation by elements of~\(L_{\Theta}^{\circ}\).  The third inclusion follows from the inclusion \(U_{\Theta}^{\opp\geq 0}U_{\Theta}^{\opp> 0} \subset U_{\Theta}^{\opp> 0}\) and the equality \(S = U_{\Theta}^{\opp> 0} L_{\Theta}^{\circ} U_{\Theta}^{>0}\) established in~(\ref{item3_thm:positive}).

\item 
Since \(S\subset G^{\geq 0}\), this follows directly from the previous item.

\item 
By~(\ref{item1bis_thm:positive}), we have \(\overline{S}\subset G_{\Theta}^{\geq 0}\).
Taking closures in~(\ref{item4_thm:positive}), we obtain, for all~\(g\in G^{\geq 0}\) \(g \overline{S} \subset \overline{S}\) and hence \(g\in \overline{S}\) since the neutral element of~\(G\) belongs to \(\overline{S}\).  This gives the reverse inclusion \(G_{\Theta}^{\geq 0}\subset \overline{S}\) and hence the equality $\overline{S} = G_{\Theta}^{\geq 0}$. 

\item 
This is an immediate consequence of the previous item, thus  $G_{\Theta}^{\geq 0} = U_\Theta^{\geq 0} \cdot L_{\Theta}^{\circ}  \cdot U_{\Theta}^{\opp \geq 0} =  U_{\Theta}^{\opp \geq 0} \cdot L_{\Theta}^{\circ} \cdot U_\Theta^{\geq 0}$. 

\item 
By (\ref{item2_thm:positive}), we have \(S=\overline{S} \cap \Omega \cap \Omega^\opp\) and the first statement follows from the equality \(\overline{S}=G_{\Theta}^{\geq 0}\).  Furthermore, since \(S= U_{\Theta}^{>0} L_{\Theta}^{\circ} U_{\Theta}^{\opp>0}\) is the interior of \(\overline{S}= U_{\Theta}^{\geq 0} L_{\Theta}^{\circ} U_{\Theta}^{\opp\geq 0}=G_{\Theta}^{\geq 0}\), we get \(S=G_{\Theta}^{>0}\).
\qedhere
\end{enumerate}
\end{proof}

Theorem~\ref{thm:decomposition} is now a direct consequence of Theorem~\ref{thm:positive}.

\begin{corollary}\label{cor:positive}
Let $g$ be an element in $G_\Theta^{>0}$. Then: 
\begin{enumerate}
\item The triple $({p}_\Theta, g\cdot {p}_\Theta, {p}_{\Theta}^{\prime})$ 
 is positive. 
\item The quadruple $({p}_\Theta, g\cdot {p}_\Theta, {p}_{\Theta}^{\prime}, g^{-1}\cdot {p}^{\prime}_{\Theta})$ is positive. 
\item The quadruple \(({p}_\Theta, g\cdot {p}_\Theta, g\cdot{p}_{\Theta}^{\prime}, {p}^{\prime}_{\Theta})\) is positive. 
\end{enumerate}
\end{corollary}

\begin{corollary}
  \label{cor_proper}
  The product maps from \(U_{\Theta}^{\geq 0}\times L_{\Theta}^{\circ} \times U_{\Theta}^{\opp\geq 0}\) and  from \(U_{\Theta}^{\opp\geq 0}\times L_{\Theta}^{\circ} \times U_{\Theta}^{\geq 0}\) into \(G\) are proper with image equal to~\(G_{\Theta}^{\geq 0}\).
\end{corollary}
\begin{proof}
  This comes from the fact that the images of these maps are equal to \(G_{\Theta}^{\geq 0}=\overline{S}\) that is closed in~\(G\).
\end{proof}

\begin{corollary}
  For all~\(g\) and~\(h\) in \(G_{\Theta}^{>0}\), one has \(G_{\Theta}^{\geq 0}\subset g^{-1} G_{\Theta}^{>0} h^{-1}\).

  Let \((g_n)_{n\in\N}\) and \((h_n)_{n\in\N}\) be sequences in \(G_{\Theta}^{>0}\) that converge to~\(e_G\), then one has \(G_{\Theta}^{\geq 0}= \bigcap_{n\in \N} g^{-1}_{n} G_{\Theta}^{>0} h^{-1}_{n}\).
\end{corollary}
\begin{proof}
  The first statement is a direct consequence of the semigroup properties~\(G_{\Theta}^{>0}\) and~\(G_{\Theta}^{\geq 0}\).

  Let \((g_n)_{n\in\N}\) and \((h_n)_{n\in\N}\) be sequences in \(G_{\Theta}^{>0}\) that converge to~\(e_G\).  By the first assertion one has the inclusion \(G_{\Theta}^{\geq 0}\subset \bigcap_{n\in \N} g^{-1}_{n} G_{\Theta}^{>0} h^{-1}_{n}\).  Let~\(g\) be in \(\bigcap_{n\in \N} g^{-1}_{n} G_{\Theta}^{>0} h^{-1}_{n}\).  For all~\(n\) in~\(\N\), let~\(f_n\) be the element of \(G_{\Theta}^{>0}\) such that \(g=g_{n}^{-1}f_n h_{n}^{-1}\).  Since \((g_n)_{n\in\N}\) and \((h_n)_{n\in\N}\) converge to~\(e_G\), the sequence \((f_n)_{n\in\N}\) converges to~\(g\).  Since \(G_{\Theta}^{\geq 0}\) is the closure of~\(G_{\Theta}^{>0}\), we get that \(g\)~belongs to~\(G_{\Theta}^{\geq 0}\).
\end{proof}

\section{Proximality in the positive semigroup}
\label{section4}

In this section we establish several properties for elements in the positive semigroup of~$G$.  

\begin{theorem}\label{thm:proximal}
Let $g \in G_{\Theta}^{>0}$. Then: 
\begin{enumerate}
  \item\label{item1_thm_proximal} The element~$g$ acts on~$\mathsf {F}_{\Theta}$ proximally.  Its attracting fixed point~\(g^+\) belongs to the standard diamond~\(D_\Theta\); its repelling fixed point~\(g^-\) belongs to the standard opposite diamond~\(D_{\Theta}^{\vee}\). 
  \item\label{item2_thm_proximal} The sequences
   \((g^n\cdot p_\Theta)_{n\in \N}\) and \((g^{-n}\cdot p_\Theta)_{n\in \N}\), as well as the sequences  \((g^n\cdot p_{\Theta}^{\prime})_{n\in \N}\) and \((g^{-n}\cdot p_{\Theta}^{\prime})_{n\in \N}\) converge to~\(g^+\) and~\(g^-\) respectively.
  \item\label{item3_thm_proximal} The quadruple $(g^-, {p}_{\Theta},g\cdot {p}_{\Theta},g^+)$ is positive.
\end{enumerate} 
\end{theorem}

This result is essentially contained in \cite[Proposition~3.14 and Proposition~3.18]{GLW} even though the positive semigroup is not explicitly mentioned in \cite{GLW}. 
We first apply Proposition~3.14 from \cite{GLW}  in Lemma~\ref{lemma_iterates_CV} to establish properties (\ref{item2_thm_proximal}) and~(\ref{item3_thm_proximal}).  Property~(\ref{item2_thm_proximal}) will be generalized in Corollary~\ref{cor_action_on_diam} below.  One could then deduce (\ref{item1_thm_proximal}) from Proposition~3.18 in \cite{GLW}, but since we will show a stronger statement (see below Proposition~\ref{prop:posprox}), we will give the details of the proof. 

In order to formulate the stronger statement, let us recall that the diagonal action of $L_{\Theta}^{\circ}$ on the direct sum of cones $\bigoplus_{\alpha\in\Theta} \mathring{c}_{\alpha}$ is transitive with stabilizers being maximal compact subgroups, so $\bigoplus_{\alpha\in\Theta} \mathring{c}_{\alpha}$ is a model for the symmetric space associated to $L_{\Theta}^{\circ}$. 
Given an element $ e = (e_\alpha)_{\alpha\in \Theta} \in \prod_{\alpha\in\Theta} \mathring{c}_{\alpha}$, and $\ell \in L_{\Theta}^{\circ}$, we have that the element $\ell\cdot e=(\Ad(\ell)e_\alpha)_{\alpha\in \Theta}$ is in $\bigoplus_{\alpha\in\Theta} \mathring{c}_{\alpha}$. We set

\[L^{e}_{\Theta} \coloneq\bigl\{ \ell \in L^{\circ}_{\Theta} \mid \ell\cdot e - e \in \bigoplus_{\alpha\in\Theta} \mathring{c}_{\alpha}\bigr\}.\]
This is a subsemigroup of \(L_{\Theta}^{\circ}\): indeed, for \(\ell'\in L_{\Theta}^{\circ}\) and \(\ell\in L^{e}_{\Theta}\), the element \(\ell' \ell\cdot e-\ell'\cdot e\) is also in \(\bigoplus_{\alpha\in\Theta} \mathring{c}_{\alpha}\) being the image of \(\ell\cdot e-e\) by~\(\ell'\) and thus \(\ell' \ell\cdot e-e= (\ell' \ell\cdot e-\ell'\cdot e) + (\ell'\cdot e- e)\) belongs to \(\bigoplus_{\alpha\in\Theta} \mathring{c}_{\alpha}\) when \(\ell\) and~\(\ell'\) are in \(L_{\Theta}^{e}\).

We have $L_{\Theta}^{g\cdot e} = g (L_{\Theta}^{e})g^{-1}$ for every~\(g\) in~\(L_{\Theta}^{\circ}\).  Hence, since $L_\Theta^{\circ}$ acts transitively on $\bigoplus_{\alpha\in\Theta} \mathring{c}_{\alpha}$, the conjugacy class of the semigroups \(L^{e}_{\Theta}\) is well-defined, and we will denote by~$L_{\Theta}^{*}$ the subsemigroup with respect to $e = (E_\alpha)_{\alpha\in \Theta}$, where $E_\alpha$ are the elements in a $\Theta$-system (Section \ref{sec_theta_prin}).

We will prove in Section~\ref{sec:continuous map} the following proposition.
\begin{proposition}\label{prop:posprox}
Let $g \in G_{\Theta}^{>0}$, then $g$~is conjugate to an element in~$L_{\Theta}^{*}$. 
\end{proposition}

When $G$~is split and $\Theta= \Delta$, $L_{\Theta}^{\circ}$ is the connected component of the Cartan subgroup, \ie $L_{\Theta}^{\circ}  = \exp(\mathfrak{a})$.  In that case $L_{\Theta}^{*}$ is the subsemigroup of elements $\exp(X)$ where $\alpha(X)>0$ for all $\alpha \in \Delta$.

In the general case, the following proposition gives an explicit description of elements in \(L_{\Theta}^{*}\) with respect to their Cartan decomposition (Section~\ref{sec_struct_levi}). 
\begin{proposition}
Let $\ell\in L_\Theta^{\circ}$. Let \(X_\ell\) be the unique element in \(\bar{\mathfrak{a}}_{\Theta}^{+} \coloneq \{X\in \mathfrak{a} \mid \forall \alpha\in \Delta\smallsetminus \Theta,\, \alpha(X)\geq 0\},\) such that \(\ell\) belongs to \(K_\Theta  \exp(X_\ell) K_\Theta \).
Then \(\ell\) is in \(L_{\Theta}^{*}\) if and only if $\alpha(X_\ell) >0$ for all $\alpha \in \Theta$. 
\end{proposition}
\begin{proof}
Let~\(\ell\) be an element of \(L_{\Theta}^{\circ}\).  There are $k_1, k_2 \in K_\Theta $ and \(X\in \overline{\mathfrak{a}}^{+}_{\Theta}\) such that $\ell = k_1 \exp(X) k_2$.  Since the maximal compact subgroup $K_\Theta $ stabilizes  $e = (E_\alpha)_{\alpha\in \Theta}$, we have that 
\[\ell\cdot e-e = k_1 \exp(X) k_2\cdot e -e = k_1\cdot ( \exp(X)\cdot e - e). \]

So $\ell\in L_\Theta^{*}$ if and only if \( \exp(X)\cdot e - e\) lies in \(\prod_{\alpha\in\Theta} \mathring{c}_{\alpha}\).

We are thus left to characterize the elements \(X \in \bar{\mathfrak{a}}_{\Theta}^{+}\) such that \(\exp(X)\) belongs to \(L_\Theta^{*}\).  To determine this we use the explicit form of the elements 
$E_\alpha \in \mathring c_\alpha$, $\alpha \in \Theta$ (see Section~\ref{sec_theta_prin}). 

There are two cases to consider.  The first case is when $\mathfrak{u}_\alpha = \mathfrak{g}_\alpha$;  $E_\alpha$~is then the root vector in $\mathfrak{g}_\alpha$ and $\mathring{c}_\alpha = \R_{>0} x_\alpha$.  Since \(\Ad(\exp(X))X_\alpha = e^{\alpha(X)}X_\alpha\), one has $\Ad(\exp(X))E_\alpha - E_\alpha \in \mathring{c}_\alpha$ if and only if \(e^{\alpha(X)}>1\) that is, if $\alpha(X)>0$. 

In the second case, the element~$E_\alpha$ is the sum of the root vector $X_\alpha$ (in $\mathfrak{g}_{\alpha}$) and of the root vectors $X_{\gamma_i}$, $i= 1, \dots, d$ in \(\mathfrak{g}_{\gamma_i}\) (see Section~\ref{sec_theta_prin}).  One has 
\[\Ad( \exp(X))E_\alpha = \sum_{i=0}^{d} \Ad( \exp(X))X_{\gamma_i} = \sum_{i=0}^{d} e^{\gamma_i(X)}X_{\gamma_i}. \]

(Recall that \(\gamma_0=\alpha\)).
Since the intersection of~\(\mathring{c}_\alpha\) with the space generated by \((X_{\gamma_i})_{i=0,\dots, d}\) is the quadrant \(\sum_{i=0}^{d}\R_{>0} X_{\gamma_i}\), one has that \(\Ad( \exp(X))E_\alpha -E_\alpha\) belongs to~\(\mathring{c}_{\alpha}\) if and only if, for all~\(i=0,\dots, d\), \(\gamma_i(X)>0\) which is equivalent to having \(\alpha(X)>0\) by the fact that \(X\)~is an element of~\(\overline{\mathfrak{a}}_{\Theta}^{+}\) and the fact that, for all~\(i\) in \(\{1,\dots, d\}\), the linear form \(\gamma_i-\gamma_0\), being a non-negative linear combination of simple roots in \(\Delta\smallsetminus \Theta\), is non-negative in restriction to~\(\overline{\mathfrak{a}}_{\Theta}^{+}\).

This proves the proposition. 
\end{proof}

\subsection{Positive elements have fixed points}

We first establish the following lemma.

\begin{lemma}\label{lemma_iterates_CV}
  Let \(g\in G^{>0}_{\Theta}\).  Then the sequences \((g^n\cdot p_\Theta)_{n\in\N}\) and \((g^{-n}\cdot p_\Theta)_{n\in\N}\) are converging.  The elements \(x= \lim_{n\to +\infty} g^n\cdot p_\Theta\) and \(y= \lim_{n\to -\infty} g^n\cdot p_\Theta\) are fixed by~\(g\) and  the quadruple \((y, p_\Theta, g\cdot p_\Theta, x)\) is positive.
\end{lemma}
\begin{proof}
    The sequences \((g^n\cdot p_\Theta)_{n\in\N}\) and \((g^{-n}\cdot p_\Theta)_{n\in\N}\) satisfy the (monotonicity) hypothesis of \cite[Proposition~3.14]{GLW}, hence converge by this proposition.  It is clear that their limits, denoted respectively~\(x\) and~\(y\), are fixed by~\(g\).  By \cite[Proposition~3.14]{GLW} again, \(x\)~belongs to the standard diamond~\(D_{\Theta}\) and \(y\)~belongs to~\(D_{\Theta}^{\vee}\).  The  quadruple \((y, p_\Theta, g\cdot p_\Theta, x)\) is thus positive (Section~\ref{sec_diamonds}).
\end{proof}

\subsection{Projections to symmetric spaces}\label{sec:projections}
The unipotent radical~\(U_{\Theta}\) of~$P_{\Theta}$ is diffeomorphic to its Lie algebra $\mathfrak{u}_{\Theta}$, \ie the maps $\exp\colon  \mathfrak{u}_\Theta\to U_\Theta$ and $\log\colon  U_\Theta \to \mathfrak{u}_\Theta$ are diffeomorphisms.
The Lie algebra $\mathfrak{u}_\Theta$ is graded by the adjoint action of~$D$ of the \(\Theta\)-principal \(\mathfrak{sl}_2\)-triple (Section~\ref{sec_theta_prin}): the degree~\(d\) part being \(\ker(\Ad(D)-d\mathrm{Id})|_{\mathfrak{u}_\Theta}\).  Its degree~\(1\) component is equal to $\bigoplus_{\alpha \in \Theta} \mathfrak{u}_\alpha$, and we have a projection $p\colon \mathfrak{u}_\Theta \to \bigoplus_{\alpha \in \Theta} \mathfrak{u}_\alpha$ on the degree~\(1\) component (see \cite[Section~10.8]{GW_pos}).  All these maps are equivariant with respect to the action of $L_{\Theta}^{\circ}$.

\begin{definition}\label{def:degree_one}
The \emph{degree~\(1\) component} of an element $u \in U_\Theta$ is the following element 
\[[u]_1\coloneq p \circ \log (u) \in  \bigoplus_{\alpha \in \Theta} \mathfrak{u}_\alpha.\]
The map \(u\mapsto [u]_1\) is called the \emph{degree~\(1\) projection}
\end{definition}
The degree~\(1\) projection is a homomorphism: for all~\(u\) and~\(v\) in \(U_\Theta\), one has \([uv]_1 = [u]_1 + [v]_1\) (this follows from the Campbell--Hausdorff formula and the fact that \(p([a,b])=0\) for every~\(a\) and~\(b\) in \(\mathfrak{u}_\Theta\)).  
Furthermore, it is equivariant with respect to the action of \(L_{\Theta}\): for all~\(\ell\) in \(L_{\Theta}\) and~\(u\) in~\(U_{\Theta}\), one has \(\ell\cdot [u]_1 = [\ell u\ell^{-1}]\).

 The parametrization of the semigroup \(U_{\Theta}^{>0}\) and Lemma~10.21 of \cite{GW_pos} give the following: 
\begin{proposition}\label{prop:symmspace}
Let $u \in U_\Theta^{> 0}$ be an element of the positive semigroup. Then the degree~\(1\) component $[u]_1$ belongs to $\bigoplus_{\alpha \in \Theta} \mathring{c}_{\alpha}$. 
\end{proposition}

\begin{remark}\label{rem.symspace}
  Another way to state this is that the map \[U_\Theta^{> 0}\to \bigoplus_{\alpha \in \Theta} \mathring{c}_{\alpha} \mid u \mapsto [u]_1\] is an $L_{\Theta}^{\circ}$-equivariant projection from the unipotent positive semigroup onto the Riemannian symmetric space of~\(L_{\Theta}^{\circ}\).
\end{remark}

Using this map, one can give a sufficient condition for an element of \(L_{\Theta}^{\circ}\) to be in \(L_{\Theta}^{*}\):

\begin{lemma}
  \label{lemma_pos_4uple_Lstar}
  Let \(\ell\) be in \(L_{\Theta}^{\circ}\).  Assume that there is an element~\(u\) in \(U_{\Theta}^{>0}\) such that the quadruple \(({p}_{\Theta}^{\prime}, u\cdot {p}_{\Theta}^{\prime}, \ell u\cdot {p}_{\Theta}^{\prime}, \mathfrak{p}_{\Theta})\) is positive.  Then \(\ell\)~is conjugate to an element of \(L_{\Theta}^{*}\), more precisely \(\ell\)~belongs to \(L_{\Theta}^{e}\) where \(e=[u]_1\).
\end{lemma}
\begin{proof}
  The hypothesis says that there is an element~\(v\) in \(U_{\Theta}^{>0}\) such that \(\ell u \ell^{-1} = vu\).  Taking the degree~\(1\) component, one gets that \(\ell\cdot [u]_1 = [u]_1 + [v]_1\) and thus,  that 
  \(\ell\cdot [u]_1 - [u]_1 \) belongs to \(\bigoplus_{\alpha\in \Theta} \mathring{c}_{\alpha}\).
\end{proof}

One can use the degree~\(1\) projection to prove that elements of \(L_{\Theta}^{*}\) are loxodromic.

\begin{proposition}\label{prop_Lthetastar_prox}
  Let~\(\ell\) be an element of \(L_{\Theta}^{*}\), then \(\ell\)~is loxodromic, more precisely the attracting fixed point of~\(\ell\) in \(\mathsf{F}_\Theta\) is \({p}_{\Theta}\) and the repelling fixed point is \({p}_{\Theta}^{\prime}\).
\end{proposition}
\begin{proof}
  By symmetry, we will only prove that the repelling fixed point of~\(\ell\) is \({p}_{\Theta}^{\prime}\).  It is clearly a fixed point of~\(\ell\) and we need to establish that the tangent action of~\(\ell\) on \(T_{{p}_{\Theta}^{\prime}} \mathsf{F}_\Theta\) is dilating.

  Since \(T_{{p}_{\Theta}^{\prime}} \mathsf{F}_\Theta\) can be \(L_{\Theta}\)-equivariantly identified with \(\mathfrak{u}_{\Theta}\) endowed with the (restriction of the) adjoint action of~\(L_\Theta\), it is thus sufficient to prove that \(\Ad(\ell)|_{\mathfrak{u}_\Theta}\) is a dilating automorphism of the Lie algebra~\(\mathfrak{u}_\Theta\).  Since the Lie algebra~\(\mathfrak{u}_\Theta\) is generated by \(\bigoplus_{\alpha\in \Theta} \mathfrak{u}_\alpha\) (this is a result of Kostant \cite{Kostant_RootLevi}, \cf \cite[Theorem~2.2]{GW_pos}), it is enough to prove that, for every~\(\alpha\) in~\(\Theta\), the automorphism \(\Ad(\ell)|_{\mathfrak{u}_\alpha}\) is dilating.
  
  By the hypothesis (and using the degree~\(1\) projection), for every~\(\alpha\) in~\(\Theta\), the element~\(x_\alpha\) in \(\mathring{c}_{\alpha}\) is such that \(y_\alpha \coloneq \Ad(\ell) x_\alpha\) belongs to \(x_{\alpha} + \mathring{c}_{\alpha}\).

  Let us introduce the two subsets \(D_\alpha = (x_{\alpha} - {c}_{\alpha}) \cap (-x_{\alpha} + {c}_{\alpha})\) and \(F_\alpha  = (y_{\alpha} - {c}_{\alpha}) \cap (-y_{\alpha} + {c}_{\alpha})\).  Those are closed and convex subsets of~\(\mathfrak{u}_\alpha\) that are of nonempty interior and invariant by multiplication by~\(-1\).  Let \(\Vert\cdot \Vert_{D_\alpha}\) and \(\Vert\cdot \Vert_{F_\alpha}\) be the norms whose unit balls are~\(D_\alpha\) and~\(F_\alpha\) respectively.  These norms on~\(\mathfrak{u}_\alpha\) satisfy the following properties: \(\Ad(\ell)|_{\mathfrak{u}_\alpha}\colon (\mathfrak{u}_\alpha, \Vert\cdot \Vert_{D_\alpha}) \to (\mathfrak{u}_\alpha, \Vert\cdot \Vert_{F_\alpha})\) is an isometry and since the closure of~\(D_{\alpha}\) is contained in~\(F_\alpha\) (this follows from the condition on~\(y_\alpha\)), there is \(k>1\) such that \(\Vert\cdot \Vert_{F_\alpha} \geq k\Vert\cdot \Vert_{D_\alpha}\).  This implies that  \(\Ad(\ell)|_{\mathfrak{u}_\Theta}\) is a dilating automorphism of the Lie algebra~\(\mathfrak{u}_\Theta\) and gives the result. 
\end{proof}

\subsection{Cross-ratios}\label{sec:crossratio}

The degree~\(1\) projection can be used to define a cross-ratio like invariant of positive quadruples. 
Given a positive quadruple $(f_1,f_2,f_3,f_4) = (p_\Theta, u_1\cdot {p}_{\Theta}^{\prime}, u_2\cdot {p}_{\Theta}^{\prime}, {p}_{\Theta}^{\prime})$ we get the two elements $[u_1]_1, [u_2]_1 \in \bigoplus_{\alpha \in \Theta} \mathring{c}_{\alpha}$, and their difference 
$[u_2]_1 -  [u_1]_1$ takes values in 
$ \bigoplus_{\alpha \in \Theta} \mathring{c}_{\alpha} \subset  \bigoplus_{\alpha \in \Theta} \mathfrak{u}_{\alpha}$. 
This does not yet give a cross-ratio function.  The Levi subgroup $L_\Theta^\circ$ acts on such four tuples, the degree~\(1\) projection is equivariant with respect to this action. The action of  $L_\Theta^\circ$ allows to take $[u_1]_1$ to $e = (x_\alpha)_{\alpha \in \Theta}$, and $[u_2]_1$ to an element $(y_\alpha)_{\alpha \in \Theta}$ such that $y_\alpha = r_\alpha x_\alpha$ when $\alpha \in \Theta\smallsetminus \{ \alpha_\Theta\}$ and $y_\alpha = x_\alpha + \sum_{i= 0}^k r_iz_i$, for $r_i>0$. This thus gives a well-defined cross ratio function 
$c(f_1,f_2,f_3,f_4)$, which is an invariant of positive quadruples up to the action of $G$. It is defined an all positive quadruples $(f_1,f_2,f_3,f_4) = g \cdot (p_\Theta, u_1\cdot {p}_{\Theta}^{\prime}, u_2\cdot {p}_{\Theta}^{\prime}, {p}_{\Theta}^{\prime})$. 

There is another geometric interpretation of this cross ratio function. 
As we can identify $\bigoplus_{\alpha \in \Theta} c^\circ_\alpha$ with the symmetric space associated to $L_\Theta^\circ$, 
the points $[u_1]_1, [u_2]_1$ can be considered as points in the symmetric space. Again, using the action of $L_\Theta^\circ$, we can move $[u_1]_1$ to the base point and $[u_2]_1$ into the Weyl chamber $\bar{\mathfrak{a}}_{\Theta}^{+}$ of the Cartan subspace of $L_\Theta^\circ$. 
Thus, the cross ratio function agrees with the vector valued distance function $d_{\mathfrak{a}_\Theta} ([u_1]_1, [u_2]_1)\in \bar{\mathfrak{a}}_{\Theta}^{+}$. 
The projection of this cross ratio function to the center $\mathfrak{z}_\Theta \subset \bar{\mathfrak{a}}_{\Theta}$ of $\mathfrak{l}_\Theta$ gives real valued cross ratio functions.

The construction of a cross-ratio for four tuples of flags is not new at all. Kantor in \cite{Kantor_crossratio} gives a very general construction of cross ratio functions of four tuples of flags.  Siegel \cite{Siegel_symplectic} has used a cross ratio function of four Lagrangian subspaces to express the Riemannian metric on the Siegel upper half space. This symplectic cross ratio function has further been studied by \cite{HartnickStrubel, BurgerPozzetti}
Further cross ratio functions have been considered in \cite{Labourie_crossratio, Kim_crossratio, Beyrer, BGLPW}. 
In  \cite{Riestenberg_Smillie} a cross-ratio function similar to the above one is used to introduce a notion of quasi-symmetric maps in the full flag variety of $\SL_n(\R)$. 

\subsection{Proximality}

Let~\(g\) be in \(G_{\Theta}^{>0}\).  By the definition of~\(G_{\Theta}^{>0}\), the element~\(g\) belongs to~\(G^\circ\).  Since we are mainly concerned in this section by the action of~\(g\) on the flag variety, we will rather work with the image of~\(g\) in the adjoint group \(\mathrm{Aut}_0(\mathfrak{g})\subset \mathrm{Aut}_1(\mathfrak{g})\).  We will denote by~\(\bar{g}\) this element of \(\mathrm{Aut}_0(\mathfrak{g})\); the actions of~\(g\) and~\(\bar{g}\) on \(\mathsf{F}_\Theta\) coincide and one has \(g^+=\bar{g}^+\) and \(g^-=\bar{g}^-\).   
  By Lemma~\ref{lemma_iterates_CV}, the sequences \((\bar{g}^n\cdot p_\Theta)_{n\in\N}\) and \((\bar{g}^{-n}\cdot p_\Theta)_{n\in\N}\) converge in~\(\mathsf{F}_\Theta\) to elements~\(x\) and~\(y\) respectively.  The quadruple \((y, p_\Theta, g\cdot p_\Theta, x)\) is positive, and \(x\) and \(y\) are fixed by~\(\bar{g}\).

Let~\(h\) be an element of~\(\mathrm{Aut}_1(\mathfrak{g})\) such that \(h\cdot y = p_{\Theta}^{\prime}\), \(h\cdot x = p_\Theta\), and \(f=h\cdot p_\Theta\) belongs to the standard diamond.  The element \(j=h\bar{g}h^{-1}\) fixes \(p_\Theta\) and \(p_{\Theta}^{\prime}\) and thus belongs to \(L_\Theta\).  Let \(k\) be a positive integer such that \(\ell=j^k\) belongs to~\(L_{\Theta}^{\circ}\).  By applying 
Lemma~\ref{lemma_iterates_CV} to \(g^k\) we get that \((p_{\Theta}^{\prime}, f, \ell\cdot f, p_\Theta)\) is also a positive quadruple.  By Lemma~\ref{lemma_pos_4uple_Lstar} we get that \(\ell\) is conjugate to an element of \(L_{\Theta}^{*}\) and by Proposition~\ref{prop_Lthetastar_prox} we get that \(\ell\) is proximal with attracting and repelling fixed points \({p}_{\Theta}\) and \(p_{\Theta}^{\prime}\) respectively.  From this we obtain that \(g^k\) (and hence~\(g\)) is proximal with attracting and repelling fixed points 
\(g^+ = x = \lim_{n\to \infty} g^n\cdot {p}_{\Theta}\) and \(g^- = y =  \lim_{n\to \infty} g^{-n} \cdot p_{\Theta}^{\prime}\) respectively.

With this Theorem~\ref{thm:proximal} is established. 

\subsection{Continuous maps from the positive semigroup}\label{sec:continuous map}

We now prove Proposition~\ref{prop:posprox}; along the way we state the continuity of several maps from the positive semigroup. Some of these statements are of interest in their own right. 

Let us explain quickly the strategy: we will prove that we can continuously conjugate elements of~\(G_{\Theta}^{>0}\) into~\(L_\Theta\) by ensuring that the fixed points of the conjugate are~\(p_\Theta\) and \(p_{\Theta}^{\prime}\); with the additional information of the positivity of a quadruple we will be able to conclude.

The arguments rely on the fact that the attracting (\resp repelling) fixed points vary continuously with the proximal elements. 

The map
\begin{align*}
  G_{\Theta}^{>0} & \longrightarrow \mathsf{F}_\Theta\\
  g &\longmapsto g^-
\end{align*}
is continuous and, for all~\(g\), the point~\(g^-\) is transverse to~\(p_\Theta\) (since the  quadruple \(( g^-, p_\Theta, g\cdot p_\Theta, g^+)\) is positive); there is thus a unique element~\(u_g\) in~\(U_\Theta\) such that \(u_g\cdot g^-= p_{\Theta}^{\prime}\).
The map \(G_{\Theta}^{>0} \rightarrow U_\Theta\mid 
  g \mapsto u_g\)
is continuous.

The map
\begin{align*}
  G_{\Theta}^{>0} & \longrightarrow \mathsf{F}_\Theta\\
  g &\longmapsto u_g\cdot g^+
\end{align*}
is continuous and, for all~\(g\), the point~\(u_g\cdot g^+\) is transverse to~\(u_g\cdot g^- = p_{\Theta}^{\prime}\).  There is thus a unique element \(v_g\) in \(U_{\Theta}^{\opp}\) such that \(v_g u_g \cdot g^+= p_\Theta\).  Furthermore, the map \(  G_{\Theta}^{>0} \rightarrow U_{\Theta}^{\opp}\mid
  g \mapsto v_g
\)
is continuous and, for all~\(g\), \(v_g u_g \cdot g^-= p_{\Theta}^{\prime}\).  For all~\(g\), the element \(v_g u_g\cdot p_\Theta\) is transverse to \(p_{\Theta}^{\prime}\) and to \({p}_{\Theta}\) and the triple \((p_{\Theta}^{\prime}, v_g u_g\cdot p_\Theta ,{p}_{\Theta})\) is positive (since the triple \((g^-, {p}_{\Theta}, g^+)\) is positive).
Since $G_\Theta^{>0}$ is connected, 
for all~\(g\) in \(G_{\Theta}^{>0}\), 
the element \(v_g u_g\cdot p_\Theta\) belongs to the standard diamond. 
There exists thus a unique element \(w_g\) in \(U_{\Theta}^{>0}\) such that \(v_g u_g\cdot p_\Theta = w_g\cdot p_{\Theta}^{\prime}\) and the map
\(  G_{\Theta}^{>0} \rightarrow U_{\Theta}^{>0}\mid
  g \mapsto w_g
\)
is continuous.  The map
\begin{align*}
  G_{\Theta}^{>0} & \longrightarrow \bigoplus_{\alpha\in \Theta} \mathring{c}_{\alpha}\\
  g &\longmapsto [w_g]_1
\end{align*}
is continuous.  

Since \(\bigoplus_{\alpha\in \Theta} \mathring{c}_{\alpha}\) is a model of the symmetric space, using the Iwasawa decomposition of $L_\Theta^{\circ}$, one gets that the group $L_\Theta^\circ\cap U_\Delta$ acts transitively on the symmetric space of $L_\Theta^{\circ}$. Thus,  there is a unique element \(s_g\) in \(L_{\Theta}^{\circ} \cap U_\Delta\) such that \([s_g w_g]_1 = (x_\alpha)_{\alpha\in \Theta}\) and the map \(g\mapsto s_g\) is continuous.  By connectedness and continuity, we get first that, for all~\(g\) in \(G_{\Theta}^{>0}\), the conjugate \(\ell_g = (s_g v_g u_g) g (s_g v_g u_g)^{-1}\) is an element of~\(L_{\Theta}^{\circ}\).  Applying Lemma~\ref{lemma_pos_4uple_Lstar} we get that, for every~\(g\) in \(G_{\Theta}^{>0}\), the element~\(\ell_g\) belongs to~\(L_{\Theta}^{*}\).

\subsection{Consequences for the action on diamonds}
As a corollary of our above consideration we also get a refinement of Lemma~\ref{lemma_iterates_CV}, which is of independent interest.

\begin{corollary}\label{cor_action_on_diam}
Let~$g$ be an element of the positive semigroup $G_\Theta^{>0}$. Then the action of~$g$ on the standard diamond~$D_\Theta$ is contracting and $\lim_{n \to \infty}g^n\cdot D_\Theta$ is reduced to a unique point, namely~$g^+$.  
\end{corollary}
\begin{proof}
  The basin of attraction for the action of~\(g\) on~\(\mathsf{F}_\Theta\) is the affine chart whose elements are the points of~\(\mathsf{F}_\Theta\) that are transverse to the repulsive fixed point~\(g^-\).  As \(g^-\) belongs to the opposite standard diamond \(D_{\Theta}^{\vee}\), the closed diamond \(\overline{D}_{\Theta}\) is contained in the basin of attraction of~\(g\).  By compactness of~\(\overline{D}_\Theta\) we get the wanted contraction property as well as the fact that the sequence \((g^n\cdot \overline{D}_\Theta)_{n\in \N}\) converges, in the Hausdorff topology, to \(\{g^+\}\).
\end{proof}

\section[Positive and non-negative parts]{Positive and non-negative parts in flag varieties}\label{sec:pos_flag}

In \cite{LusztigPosRed, Lusztig} Lusztig introduced, in the context of total positivity, the positive and non-negative part of flag varieties. These have since been extensively studied \cite{Rietsch_closure, Rietsch_decomposition, RietschWilliams_complex, RietschWilliams_discrete, Williams_shelling, GalashinKarpLam, GalashinKarpLam_Gr}. Particular interest has been on the positive Grassmannian \cite{Postnikov_total}, which plays a major role in the study of scattering amplitudes in physics due to its relation to the amplituhedron, \cite{Arkani_etal_book, ArkanihamedTrnka}. 

 In this section we introduce and investigate the positive and non-negative parts of flag varieties for groups $G$ admitting a positive structure relative to $\Theta$.

\subsection{The positive part in $\mathsf{F}_{\Theta}$}
We first consider the case of the flag variety $\mathsf{F}_{\Theta}$,  for $\Theta \subset \Delta$ the subset defining the positive structure in $G$. 

\begin{definition}\label{def:pos_part}
The \emph{positive part} $\mathsf{F}_{\Theta}^{>0}$ of the flag variety $\mathsf{F}_{\Theta}$ is 
\[\mathsf{F}_{\Theta}^{>0} \coloneq  G_{\Theta}^{>0}\cdot {p}_{\Theta} = \{ u\cdot {p}_{\Theta} \mid u \in G_{\Theta}^{>0}\}.  
\]
The \emph{non-negative part} of the flag variety  $\mathsf{F}_{\Theta}^{\geq0}$ is the closure of $\mathsf{F}_{\Theta}^{>0}$ in~$\mathsf{F}_{\Theta}$. 
\end{definition}
\begin{lemma}\label{lem:pos_diamond}
The positive part $\mathsf{F}_{\Theta}^{>0}$  of $\mathsf{F}_{\Theta}$ is equal to the diamond~$D_\Theta$.  The non-negative part  $\mathsf{F}_{\Theta}^{\geq0}$ is the closure of the diamond~$\overline{D}_{\Theta}$.  In particular, we have that
$\mathsf{F}_{\Theta}^{>0} = G_{\Theta}^{>0}\cdot p_{\Theta}^{\prime} =U_{\Theta}^{>0}\cdot p_{\Theta}^{\prime}$.
\end{lemma}
\begin{proof}
The equality \(\mathsf{F}_{\Theta}^{>0} = D_\Theta\) follows from the decomposition $G_{\Theta}^{>0} = U_{\Theta}^{\opp >0} L_{\Theta}^{\circ} U_\Theta^{> 0}$ and the fact that $L_{\Theta}^{\circ} U_\Theta^{> 0}$ stabilizes~\({p}_\Theta\).  Since the equality \(D_\Theta = U_{\Theta}^{>0}\cdot p_{\Theta}^{\prime}\) holds, since $G_{\Theta}^{>0} = U_{\Theta}^{>0} L_{\Theta}^{\circ} U_\Theta^{\opp >0}$, and since \(L_{\Theta}^{\circ} U_\Theta^{\opp >0}\) stabilizes \(p_{\Theta}^{\prime}\), one also gets that
$\mathsf{F}_{\Theta}^{>0} = G_{\Theta}^{>0}\cdot p_{\Theta}^{\prime}$. 
\end{proof}

The following corollaries are immediate consequences of statements in the previous section.

\begin{corollary}
Let $g \in G_{\Theta}^{>0}$. Then there exists a unique fixed point in $\mathsf{F}_{\Theta}^{\geq 0}$ for the action of~$g$.  This fixed point is the attracting fixed point~$g^+$.
\end{corollary}

\begin{corollary}\label{cor:conncomp}
  Let \(\mathcal{O}\subset \mathsf{F}_\Theta\) be the affine chart with respect to~\(p_\Theta\) (\ie the set of elements in \(\mathsf{F}_{\Theta}\) that are transverse to~\(p_\Theta\)) and let \(\mathcal{O}' \subset \mathsf{F}_\Theta\) be the affine chart with respect to~\(p_{\Theta}^{\prime}\).
Then the set $\mathsf{F}_{\Theta}^{>0}$ is a connected component of $\mathcal{O} \cap \mathcal{O}'$.
\end{corollary}

\subsection{The positive part in other flag varieties}
\label{sec:pos_part_other}

Let $\Theta' \subset \Delta$ be any subset and ${\mathsf{F}}_{\Theta'}$ the corresponding flag variety. 
In Section~\ref{sec_flag_var}, we introduced the \(G^\circ\)-equivariant projections~\(\pi_{\Theta}^{\Xi}\colon \mathsf{F}_\Xi \to \mathsf{F}_\Theta\) between flag varieties.

When \(S\)~is a semigroup of~\(
G^\circ\) and \(x\)~is a point of~\(\mathsf{F}_{\Theta'}\) (\resp \(A\)~is a subset of~\(\mathsf{F}_{\Theta'}\)), we call (by a slight abuse of terminology),  the subset \(S\cdot x = \{ s\cdot x\}_{s\in S}\) (\resp \(S\cdot A=\{s\cdot x\}_{s\in S, x\in A}\)) the \(S\)-orbit of~\(x\) (\resp the \(S\)-orbit of~\(A\)).  

\begin{definition}
The positive part of the flag variety ${\mathsf{F}}_{\Theta'}$ is
\[{\mathsf{F}}_{\Theta'}^{>0} \coloneq   \{ g\cdot x\mid g \in G_{\Theta}^{>0}, x\in \pi_{\Theta'}((\pi_{\Theta})^{-1}(p_{\Theta}))\}, 
\]
where $\pi_\Theta\colon {\mathsf{F}}_{\Delta} \rightarrow {\mathsf{F}}_{\Theta}$ and $\pi_{\Theta'}\colon {\mathsf{F}}_{\Delta} \rightarrow {\mathsf{F}}_{\Theta'}$ are the natural projections. 
\end{definition}

Let $\Xi$ be any subset of~\(\Delta\) containing $\Theta \cup \Theta'$, by the compatibilities properties of the projections between flag varieties (\cf Section~\ref{sec_flag_var}), one also has
\[{\mathsf{F}}_{\Theta'}^{>0} = \{ g\cdot x\mid g \in G_{\Theta}^{>0}, x\in \pi_{\Theta'}^{\Xi}((\pi_{\Theta}^{\Xi})^{-1}(p_{\Theta}))\}.
\]

We have the following description of the positive part of the flag variety ${\mathsf{F}}_{\Theta'}$.

\begin{proposition}
The following holds: 
\begin{enumerate}[leftmargin=*]
\item The positive part ${\mathsf{F}}_{\Theta'}^{>0}$ of the flag variety ${\mathsf{F}}_{\Theta'}$ is the $G_{\Theta}^{>0}$-orbit of the subvariety $\pi_{\Theta'}((\pi_{\Theta})^{-1}(p_{\Theta})) = \pi^{\Xi}_{\Theta'}((\pi^{\Xi}_{\Theta})^{-1}(p_{\Theta})) = L_{\Theta}^{\circ}\cdot p_{\Theta'}$. 
\item The positive part ${\mathsf{F}}_{\Theta'}^{>0}$ of the flag variety ${\mathsf{F}}_{\Theta'}$ is the $G_{\Theta}^{>0}$-orbit of any point $q \in \pi_{\Theta'}((\pi_{\Theta})^{-1}(p_{\Theta}))$. 
\item The positive part ${\mathsf{F}}_{\Theta'}^{>0}$ of the flag variety ${\mathsf{F}}_{\Theta'}$ is the $U_{\Theta}^{\opp >0}$-orbit of the subvariety $\pi_{\Theta'}((\pi_{\Theta})^{-1}(p_{\Theta}))$. More precisely, it is the disjoint union of the $U_{\Theta}^{\opp >0}$-orbits of points in $\pi_{\Theta'}((\pi_{\Theta})^{-1}(p_{\Theta}))$. 
\end{enumerate}
\end{proposition}

\begin{proof}
The first statement is immediate from the definition and the fact that \(L_{\Theta}^{\circ}\)~acts transitively on the fibers of~\(\pi_\Theta\). The second statement and the first assertion in the third statement follow from the properties stated in Section~\ref{sec:maps}, the decomposition $G_{\Theta}^{>0} = U_{\Theta}^{\opp >0} L_{\Theta}^{\circ} U_\Theta^{> 0}$, and the fact that the action of~$L_{\Theta}^{\circ}$ on $\pi^{\Xi}_{\Theta'}((\pi^{\Xi}_{\Theta})^{-1}(p_{\Theta}))$ is transitive; in particular,  \(\pi^{\Xi}_{\Theta'}((\pi^{\Xi}_{\Theta})^{-1}(p_{\Theta})) = L_{\Theta}^{\circ}\cdot p_{\Theta'}\).

For the second assertion in the third statement, we show that for any $u \in   U_{\Theta}^{\opp >0}$ and any $v \in \pi_{\Theta'}((\pi_{\Theta})^{-1}(p_{\Theta}))$, $u \cdot v \notin \pi_{\Theta'}((\pi_{\Theta})^{-1}(p_{\Theta}))$. Let us recall that for every $u\in U_{\Theta}^{\opp >0}$, the point $u\cdot p_\Theta$ is transverse to $p_\Theta$. If there were an element $v\in \pi_{\Theta'}((\pi_{\Theta})^{-1}(p_\Theta))$, and an element $u \in U_{\Theta}^{\opp >0}$ with $u \cdot v \in \pi_{\Theta'}((\pi_{\Theta})^{-1}(p_\Theta))$, we would have 
that $\pi_{\Theta'}((\pi_{\Theta})^{-1}(p_\Theta)) \cap \pi_{\Theta'}((\pi_{\Theta})^{-1}(u\cdot p_\Theta)) \neq \emptyset$, since by the $G$-equivariance of the projection maps $u\cdot v \in \pi_{\Theta'}((\pi_{\Theta})^{-1}(u\cdot p_\Theta))$. But this is a contradiction, since $p_\Theta$ and $u \cdot p_\Theta$ correspond to opposite parabolic subgroups. 
\end{proof}

\subsection{Comparison for different positive structures}\label{sec:twostructures}
As explained in \cite{GW_pos} a semisimple Lie group admits a positive structure if and only if each of its simple factors admits a positive structure. 
For the following discussion we restrict to simple Lie groups.  The simple split real Lie groups admitting two different positive structures are those with Dynkin diagram $B_n$, $C_n$, or $F_4$; there are then one positive structure with respect to~$\Delta$ and another with respect to a proper subset $\Theta \subset \Delta$. 
Thus they have two positive semigroups $G_\Delta^{>0}$ and $G_{\Theta}^{>0}$. We denote the corresponding positive parts of the flag variety $\mathsf{F}_{\Theta'}$ by 
${\mathsf{F}}_{\Theta', \Delta}^{>0}$ and ${\mathsf{F}}_{\Theta'}^{>0}$ respectively. 

\begin{theorem}
Let~$G$ be a simple split real Lie group admitting a positive structure with respect to a proper subset $\Theta\subset \Delta$. 
Let $\Theta'\subset\Delta$ be a subset and ${\mathsf{F}}_{\Theta'}$ the corresponding flag variety. Then we have a proper inclusion 
${\mathsf{F}}_{\Theta', \Delta}^{>0} \subset {\mathsf{F}}_{\Theta'}^{>0}$.
\end{theorem}
\begin{proof}
Since ${\mathsf{F}}_{\Delta, \Delta}^{>0}$ is a $U_{\Delta}^{\opp >0}$-orbit of $p_\Delta$, we have that ${\mathsf{F}}_{\Theta', \Delta}^{>0}$
is the $U_{\Delta}^{\opp >0}$-orbit of $\pi_{\Theta'}({p}_\Delta)=p_{\Theta'}$.  On the other hand ${\mathsf{F}}_{\Theta'}^{>0}$ is the $U_{\Theta}^{\opp> 0} L_{\Theta}^{\circ}$-orbit of ${p}_{\Theta'}\subset \pi_{\Theta'}((\pi_{\Theta})^{-1}({p}_\Theta))$. Since $U_{\Delta}^{\opp >0} \subset U_{\Theta}^{\opp > 0} L_{\Theta}^{\circ}$ this gives the inclusion.  Furthermore, there are clearly $\ell\in  L_{\Theta}^{\circ}$ such that $\ell\cdot p_{\Theta'} \notin U_{\Delta}^{\opp >0}\cdot p_{\Theta'}$.
\end{proof}

\subsection{The non-negative part of a flag variety} 
In the case of total positivity, it has been shown in \cite{GalashinKarpLam, GalashinKarpLam_Gr} that the non-negative part of any partial flag variety is homeomorphic to a closed ball. This is not true anymore in the case of positivity with respect to a proper subset $\Theta$ of $\Delta$.  The next statement gives a topological description of the positive and non-negative parts in \(\mathsf{F}_{\Theta'}\) when \(\Theta'\) is either contained in or contains~\(\Theta\). 

\begin{theorem}\label{thm:nonneg_ball}
Let~$G$ be a simple Lie group with a positive structure with respect to $\Theta \subset \Delta$. Let $\Theta' \subset \Delta$ be a subset and  $\mathsf{F}_{\Theta'}^{\geq 0}$ be the non-negative part in the flag variety $\mathsf{F}_{\Theta'}$. Then the following hold: 
\begin{enumerate}
\item\label{item1_thm:nonneg_ball} If $\Theta'$~is contained in~$ \Theta$, then $\mathsf{F}_{\Theta'}^{\geq 0}$ (\resp \(\mathsf{F}_{\Theta'}^{>0}\)) is homeomorphic to a closed (\resp open) ball. 
\item\label{item2_thm:nonneg_ball} If $\Theta'$ contains~$\Theta$, then $\mathsf{F}_{\Theta'}^{\geq 0}$ (\resp \(\mathsf{F}_{\Theta'}^{>0}\)) fibers, via the projection \(\pi^{\Theta'}_{\Theta}\), over the closed ball \(\overline{D}_\Theta\) (\resp the open ball~\(D_\Theta\)) with fiber $K_{\Theta}/K_{\Theta'}$, where $K_{\Theta}$ and $K_{\Theta'}$ are the maximal compact subgroups of the Levi subgroups $L_{\Theta}^{\circ}$ and $L_{\Theta'}^{\circ}$ respectively.  In particular, the non-negative part $\mathsf{F}_{\Theta'}^{\geq 0}$ (\resp the positive part \(F_{\Theta'}^{>0}\)) is homotopy equivalent to $K_{\Theta}/K_{\Theta'}$
\end{enumerate}
\end{theorem}

\begin{remark}
In the general case, when $\Theta'$ is neither contained in, nor contains $\Theta$, the non-negative part $\mathsf{F}_{\Theta'}^{\geq 0}$ is the projection of the non-negative part $\mathsf{F}_{\Theta' \cup \Theta}^{\geq 0}$, thus the projection of a fiber bundle as in~(\ref{item2_thm:nonneg_ball}). So one can not expect this to be homeomorphic to a closed ball. The examples for $G= \Sp_{2n}(\R)$ discussed in the next section illustrate this clearly.
\end{remark}

In the case of total positivity, \ie when $\Theta = \Delta$, statement~(\ref{item1_thm:nonneg_ball}) was proved in \cite{GalashinKarpLam, GalashinKarpLam_Gr}.  The proof presented here is new in the total positive case and is simpler than the one in \cite{GalashinKarpLam, GalashinKarpLam_Gr}.   We only use 
 the $\Theta$-principal $\mathfrak{sl}_2(\R)$, associated to the positive structure, and its action
  on the flag variety. 
  
\medskip

Point~(\ref{item2_thm:nonneg_ball}) of Theorem~\ref{thm:nonneg_ball} follows from the description of the non-negative (\resp positive) part (Section~\ref{sec:pos_part_other}) and, for the second assertion, from the point~(\ref{item1_thm:nonneg_ball}) applied to \(\Theta'=\Theta\).

\smallskip

Let us address point~(\ref{item1_thm:nonneg_ball}) of the theorem ; let thus~\(\Theta'\) be a subset of~\(\Theta\).

Since the map \(\pi_{\Theta'}^{\Theta}\colon \mathsf{F}_\Theta \to \mathsf{F}_{\Theta'}\) is a submersion (and hence is open) and is proper, we have that \(\mathsf{F}_{\Theta'}^{>0}= \pi_{\Theta'}^{\Theta} (\mathsf{F}_{\Theta}^{>0})\) is open and that \(\mathsf{F}_{\Theta'}^{\geq 0}= \pi_{\Theta'}^{\Theta} (\mathsf{F}_{\Theta}^{\geq 0})\).  Furthermore, \(\mathsf{F}_{\Theta'}^{\geq 0}\) is the closure of \(\mathsf{F}_{\Theta'}^{> 0}\) and \(\partial F_{\Theta'}^{\geq 0} = \mathsf{F}_{\Theta'}^{\geq 0} \smallsetminus F_{\Theta'}^{> 0}\).

For all~\(t\) in~\(\R\), set \(h_t =\exp(tE+tF)\in G^\circ\) (where \(E\) and~\(F\) belong to the \(\Theta\)-principal \(\mathfrak{sl}_2\)-triple, \cf Section~\ref{sec_theta_prin}).  Then \((h_t)_{t\in \R}\) is a \(1\)-parameter subgroup of~\(G\).  By properties of the \(\Theta\)-principal \(\mathfrak{sl}_2\), one has that, for all~\(t\neq 0\), the element \(h_t\) acts proximally on~\(\mathsf{F}_\Theta\) and, denoting by \(p^+\) and \(p^-\) the elements of \(\mathsf{F}_\Theta\) that are the common attracting and repelling fixed points of~\(h_t\) for \(t>0\), one has that the quadruple \((p_\Theta, p^+, p_{\Theta}^{\prime}, p^-)\) is positive.  

Let \(r=\pi_{\Theta'}^{\Theta}(p^+)\) so that \(r\)~is the attracting fixed point of~\(h_t\) in \(\mathsf{F}_{\Theta'}\), for all positive~\(t\).  We will prove that \(\partial \mathsf{F}_{\Theta'}^{\geq 0}\) is homeomorphic to a sphere and that the map
\begin{align*}
  \psi \colon [ 0, +\infty\mathclose{[} \times \partial \mathsf{F}_{\Theta'}^{\geq 0} & \longrightarrow \mathsf{F}_{\Theta'}^{\geq 0} \smallsetminus \{r\}\\
  (t, x) & \longmapsto h_t\cdot x
\end{align*}
is a homeomorphism.  This will enable us to conclude that \(\mathsf{F}_{\Theta'}^{\geq 0}\) is homeomorphic to the cone over a sphere and is then homeomorphic to a closed ball and, taking interiors, that \(\mathsf{F}_{\Theta'}^{> 0}\) is homeomorphic to an open ball.

Let us prove first the injectivity of~\(\psi\); this follows from the property that, for all \(t>s\), one has
\begin{equation}
  \label{eq:nestedness_ht}
  h_t \cdot \mathsf{F}_{\Theta'}^{\geq 0} \subset h_s \cdot \mathsf{F}_{\Theta'}^{> 0}.
\end{equation}
In order to establish~\eqref{eq:nestedness_ht}, as \(\mathsf{F}_{\Theta'}^{\geq 0}\) and \(\mathsf{F}_{\Theta'}^{> 0}\) are the homeomorphic images of \(\mathsf{F}_{\Theta}^{\geq 0}\) and \(\mathsf{F}_{\Theta}^{> 0}\) by the equivariant map \(\pi_{\Theta'}^{\Theta}\), it is enough to have the inclusion \(h_t \cdot \mathsf{F}_{\Theta}^{\geq 0} \subset h_s \cdot \mathsf{F}_{\Theta}^{> 0}\); but this holds by the nestedness properties of diamonds since \(h_t \cdot \mathsf{F}_{\Theta}^{> 0}\) is the diamond with extremities \((h_t\cdot p_\Theta, h_t\cdot p_{\Theta}^{\prime})\), \(h_s \cdot \mathsf{F}_{\Theta}^{> 0}\) is the diamond with extremities \((h_s\cdot p_\Theta, h_s\cdot p_{\Theta}^{\prime})\), and the quadruple \((h_s\cdot p_\Theta, h_t\cdot p_\Theta, h_t\cdot p_{\Theta}^{\prime}, h_s\cdot p_{\Theta}^{\prime})\) is positive.

\smallskip

We now establish
that \(\partial \mathsf{F}_{\Theta'}^{\geq 0}\) is homeomorphic to a ball.  For this we will work in~\(\mathcal{O}\) the basin of attraction of~\(h_t\) for \(t>0\).  The open set~\(\mathcal{O}\) is an affine chart containing \(\mathsf{F}_{\Theta'}^{\geq 0}\), and hence~\(r\), and using the parametrization of Section~\ref{sec_flag_var}, we can consider that \(\mathcal{O}\)~is a vector space whose origin is~\(r\) and where \(h_t\)~is a linear contracting transformation.  Let \(S\subset \mathcal{O}\) be the unit sphere for some norm on this vector space (for definiteness, take a Euclidean ball with respect to a basis diagonalizing the action of~\(h_t\)).  Then, for all~\(x\) in \(\mathcal{O}\smallsetminus \{r\}\), the intersection of~\(S\) with \(\{h_t\cdot x\}_{t\in \R}\) (\resp the intersection of \(\partial \mathsf{F}_{\Theta'}^{\geq 0}\) with \(\{h_t\cdot x\}_{t\in \R}\)) consists in exactly one point and this intersection varies continuously with~\(x\) (for the intersection with  \(\partial \mathsf{F}_{\Theta'}^{\geq 0}\), one uses again the nestedness property~\eqref{eq:nestedness_ht}).  We thus get continuous maps \(S\to \partial \mathsf{F}_{\Theta'}^{\geq 0}\) and  \(\partial \mathsf{F}_{\Theta'}^{\geq 0}\to S\) that are obviously inverse one from each other.  As a conclusion, \(\partial \mathsf{F}_{\Theta'}^{\geq 0}\) is homeomorphic to a sphere.

\smallskip

The surjectivity of~\(\psi\) is then a consequence of the fact that~\(r\) is the attracting fixed point of~\(h_t\) (\(t>0\)).

\section{Flag varieties for the symplectic group}\label{sec:symplectic}

In this section we give an explicit description of the positive and non-negative part of symplectic flag varieties when we endow the symplectic group $\Sp_{2n}(\R)$ with the positive structure with respect to $\Theta  = \{\alpha_n\}$. 

\subsection{The positive semigroup}
Let $\omega$ be the standard symplectic form on $\R^{2n}$: if \((x_1, \dots, x_n, y_1, \dots, y_n)\) are the coordinates on~\(\R^{2n}\), then \(\omega=\sum_{i=1}^{n}dx_i\wedge dy_i\).  The symplectic group $\Sp_{2n}(\R)$ consists of matrices 
$ g = \bigl(\begin{smallmatrix} A &B\\C&D\end{smallmatrix}\bigr) \in \GL_{2n}(\R)$ such that $D^T A - B^T C= Id$, $C^T A = A^T C$, $D^TB = B^T D$.

The Levi subgroup \(L_\Theta^\circ\) is the subgroup consisting of matrices \(\bigl(\begin{smallmatrix} A &0\\0& {A^T}^{-1}\end{smallmatrix}\bigr)\), where $A \in \mathrm{GL}(n,\R)$ with  $\det(A) >0$. 
The unipotent subgroup $U_\Theta$ consists of the matrices \(\bigl(\begin{smallmatrix} \Id&  M\\0&\Id\end{smallmatrix}\bigr) \), where $M \in \mathrm{Sym}_{n}(\R)$ is a symmetric matrix. Elements in the  positive unipotent subgroup 
\(U_\Theta^{>0}\) are those where $M \in \mathrm{Sym}^{+}_{n}(\R)$ is positive definite. Similarly, \(U_{\Theta}^{\opp} \) consists of all matrices \(\bigl(\begin{smallmatrix} \Id &0\\N &\Id\end{smallmatrix}\bigr)\) with $N \in \mathrm{Sym}_{n}(\R)$, 
and elements in the positive semigroup \(U_\Theta^{\opp, > 0}\) are those where $N \in \mathrm{Sym}^{+}_{n}(\R)$ is positive definite.

Thus, the decomposition $G_{\Theta}^{>0} = U_{\Theta}^{> 0} \cdot L_{\Theta}^{\circ} \cdot U_{\Theta}^{\opp > 0}$ implies that an element $g = \bigl(\begin{smallmatrix} A &B\\C&D\end{smallmatrix}\bigr) \in \Sp_{2n}(\R)$ is in the positive semigroup $\Sp_{2n}^{>0}(\R)$ 
if and only if $\det(A) >0$ and $ A^{-1} B \in \mathrm{Sym}^{+}_{n}(\R)$, $ C A^{-1} \in  \mathrm{Sym}^{+}_{n}(\R)$,  
where $\mathrm{Sym}^{+}_{n}(\R)$ denotes the space of symmetric positive definite matrices. 
In fact $g \in \Sp_{2n}(\R)^{>0} $ is then $$g= \begin{pmatrix} \Id &0\\C A^{-1} &\Id\end{pmatrix} \cdot \begin{pmatrix} A &0\\0& {A^T}^{-1}\end{pmatrix}\cdot \begin{pmatrix} \Id& A^{-1} B\\0&\Id\end{pmatrix}.$$

\subsection{The space of Lagrangians} 
The flag variety $\mathsf{F}_{\Theta}$ is the space of Lagrangian subspaces $\mathrm{Lag}(\R^{2n})$. The Lagrangian subspaces corresponding to $P_{\Theta}$ and $P_{\Theta}^{\opp}$ are denoted by $L_0$ and $L_\infty$ respectively. 

Since $L_0$ and $L_\infty$ are transverse, $\R^{2n}  = L_0 \oplus L_\infty$. Any Lagrangian~$L$ which is transverse to~$L_\infty$ can be written as graph of a linear map $T_L\colon L_0 \mapsto L_\infty$. Since $L_0, L, L_\infty$ are all Lagrangian subspaces, for all~$v$ and~$w$ in~$L_0$, we have 
$0 = \omega(v+T_L(v), w+T_L(w)) = \omega(v, T_L(w)) - \omega (w, T_L(v))$, thus $T_L$ is symmetric and defines a quadratic form $q_L(v) = \omega(v, T_L(v))$ on $L_0$. 
The standard diamond  $\mathsf{F}_{\Theta}^{>0} = D_\Theta$ with extremities~$L_0$ and~$L_\infty$ consists of all Lagrangian~$L$ for which the quadratic form~$q_L$ on~$L_0$ is positive definite. The opposite diamond consists of those Lagrangians for which the quadratic form is negative definite. 

\begin{remark}
The quadratic form gives rise to an invariant of triples of Lagrangians, the  classical Maslov index $\mu(L_0, L, L_\infty)$, defined as the signature of this quadratic form.  In fact it can be defined without assuming that the Lagrangians are transverse. 
When the Maslov index is equal to~$n$, then $q_L$~is positive definite.
\end{remark} 

The map \(M\mapsto \bigl(\begin{smallmatrix}
\Id & M \\ 0 & \Id
\end{smallmatrix} \bigr)\cdot L_\infty\) gives an identification of $\mathrm{Sym}_{n}(\R)$ with  the affine chart of $\mathrm{Lag}(\R^{2n})$ with respect to $L_\infty$, where $L_0$ corresponds to the zero matrix.  In this affine chart $D_\Theta$~is identified with the cone of positive definite symmetric matrices $\mathrm{Sym}^{+}_{n}(\R)$.

To give an explicit description of the non-negative part $\mathsf{F}_{\Theta}^{\geq 0}$ we choose a different affine chart.  We pick an affine chart with respect to a Lagrangian $L$ in the diamond opposite to $D_\Theta$ and identify this chart with $\mathrm{Sym}_{n}(\R)$ in such a way that $L_0$ corresponds to the zero matrix and $L_\infty$ corresponds to the identity matrix. 

In this chart $\overline{D}_{\Theta}$ has the following description 
\[\overline{D}_{\Theta}  = \{ M\in \mathrm{Sym}_{n}(\R)\mid  M \in \mathrm{Sym}^{+}_{n}(\R), \, \Id-M \in \mathrm{Sym}^{+}_{n}(\R)\} \]

This gives explicit equations for $\mathsf{F}_{\Theta}^{\geq 0}$.  For a more combinatorial description of $\mathsf{F}_{\Theta}^{\geq 0}$ we refer the reader to \cite{Xie}. 

\subsection{Symplectic flag varieties}
Flag varieties associated to the symplectic group are flag varieties of isotropic flags. 
Let \(\boldsymbol{k}=(k_1, \dots, k_\ell)\) be a \(\ell\)-tuple of integers such that $0<k_1 < \cdots < k_\ell \leq n$ and denote by $\mathsf{F}^{\Sp}_{\mathbf{k}}(\R^{2n})$ the space of isotropic (partial) flags \((E_1, \dots, E_\ell)\) in $\R^{2n}$.  The space of Lagrangian subspaces is denoted  $\mathrm{Lag}(\R^{2n})$.  When \(k_\ell<n\) and letting \(\boldsymbol{k'}=(k_1, \dots, k_\ell,n)\), 
there  natural  projections: 
$\pi^{\boldsymbol{k'}}_{n}\colon \mathsf{F}^{\Sp}_{\boldsymbol{k'}}(\R^{2n}) \to \mathrm{Lag}(\R^{2n})$,
and 
$\pi^{\boldsymbol{k'}}_{\boldsymbol{k}}\colon \mathsf{F}^{\Sp}_{\boldsymbol{k'}}(\R^{2n}) \to \mathsf{F}^{\Sp}_{\boldsymbol{k}}(\R^{2n})$,

The subset $\pi^{\boldsymbol{k'}}_{\boldsymbol{k}} ((\pi^{\boldsymbol{k'}}_{n})^{-1}(L_0))$ then consists of all the flags that are contained in $L_0$. In particular, the Levi subgroup $L^{\circ}$ acts transitively on this set. 

\begin{corollary}\label{cor:symplectic}
The positive part $\mathsf{F}^{\Sp}_{\boldsymbol{k}}(\R^{2n})^{>0}$ consists precisely of those isotropic flags \((E_1, \dots, E_\ell)\), where the subspace~\(E_\ell\) (and hence any $E_i$, $1\leq i\leq \ell$) is the graph of a positive definite symmetric map from a subspace of~$L_0$ to~$L_\infty$.
\end{corollary}

\subsection{The positive part of projective space} 
We give a more explicit description of the positive and non-negative isotropic flag varieties. 
For this look at projective space, the space of (isotropic) lines which is $ \mathbb{P}^{2n-1}(\R)$. 

 The two transverse Lagrangians $L_0$ and $L_\infty $ fixed by $P_{\Theta}$ and $P_{\Theta}^{\opp}$ give a decomposition $\R^{2n} = L_0 \oplus L_\infty$. Given a non-zero vector $v$ write $v = v_0 + v_\infty$ with respect to this decomposition. 

The symplectic form defines a quadratic form $q\colon \R^{2n} \to \R$, $ q(v) = \omega(v_0, v_\infty)$. This quadratic form is non-degenerate of signature $(n,n)$. 
We denote
\begin{align*}
  \mathcal{H}^{>0} &= \mathbb{P}(\{ v\in \R^{2n} \mid q(v) >0\}),\\ 
  \mathcal{H}^{\geq 0} &= \mathbb{P}(\{ v\in \R^{2n}\smallsetminus \{0\} \mid q(v) \geq 0\}),\\
  \intertext{and}  
  \mathcal{H}^{<0} &= \mathbb{P}(\{ v\in \R^{2n} \mid q(v) <0\}).
\end{align*}

Corollary~\ref{cor:symplectic} gives that $\mathsf{F}^{\Sp}_{1}(\R^{2n})^{>0}  = \mathcal{H}^{>0}$, and 
$\mathsf{F}^{\Sp}_{1}(\R^{2n})^{\geq 0} = \mathcal{H}^{\geq 0}$. In particular the non-negative and the non-positive part of projective space cover $\mathbb{P}^{2n-1}(\R)$, their intersection is the hypersurface $\mathcal{H} = \mathbb{P}(\{ v\in \R^{2n}\smallsetminus \{0\} \mid q(v) =0\})$, which is a quadratic hypersurface. If we pick a symplectic basis $e_1, \dots, e_n, f_1, \dots, f_n$ such that $L_0$ is the span of the $e_i$ and $L_\infty$ is the span of the $f_i$. We can write a vector $v = \sum_{i= 1}^n a_i e_i + \sum_{i=1}^{n} b_i f_i$, then $\mathcal {H}$ is given by the equation $\sum_{i=1}^n a_i b_i = 0$. Similarly, $\mathcal{H}^{\geq 0} $ is given by $\sum_{i=1}^n a_i b_i \geq 0$.

Thus, we obtain the following more geometric description of the positive and non-negative parts of symplectic flag varieties, which improves on Corollary~\ref{cor:symplectic}:
\begin{corollary}
%
The positive part $\mathsf{F}^{\Sp}_{\boldsymbol{k}}(\R^{2n})^{>0}$ (\resp the non-negative part $\mathsf{F}^{\Sp}_{\boldsymbol{k}}(\R^{2n})^{\geq 0}$) consists of those isotropic flags \((E_1, \dots, E_\ell)\), where  $\mathbb{P}(E_\ell) \subset \mathcal{H}^{>0}$ (\resp $\mathbb{P}(E_\ell) \subset \mathcal{H}^{\geq 0}$).
\end{corollary}

Let us compare this positive structure on $\mathsf{F}^{\Sp}_{1}(\R^{2n}) = \mathbb{P}(\R^{2n}) $  with the positive part of $\mathbb{P}(\R^{2n})$ with respect to the total positive structure with respect to the standard opposite minimal parabolic subgroups associated to the symplectic basis $e_1, \dots, e_n, f_1, \dots, f_n$. 

A point in $\mathbb{P}(\R^{2n})$ spanned by $v = \sum_{i= 1}^n a_i e_i + \sum_{i=1}^{n} b_i f_i$ is contained in the positive part  with respect to the total positive structure if and only if $a_i>0$ and $b_i >0$ for all~$i$, whereas it is in $\mathsf{F}^{\Sp}_{1}(\R^{2n})^{>0}$ if $\sum_{i=1}^n a_i b_i >0$. 

\begin{remark}
The hypersurface $\mathcal{H}$ and the half-spaces $\mathcal{H}^{>0}$, $\mathcal{H}^{<0}$ have been introduced by the authors some years ago to study the geometric structures associated to maximal representations of surface groups into $\Sp_{2n}(\R)$. These geometric structures can be decomposed along hypersurface like $\mathcal{H}^0$, which allows one to get a good control on the topology of quotients of domain of discontinuity. An important feature is the nestedness of the half-spaces $\mathcal{H}^{>0}$ for positive quadruples of Lagrangians, which in the perspective here follows immediately from the nestedness of diamonds. This nestedness property  was used by Burelle and Treib \cite{Burelle_Treib} to construct maximal Schottky groups. 
\end{remark}


\bibliographystyle{amsalpha}
\bibliography{geometric-positivity.bib}

\end{document}